\documentclass[12pt,a4paper]{article}

\usepackage{lmodern}
\usepackage[top=2cm, bottom=2cm, left=2cm, right=2cm]{geometry}
\usepackage{tikz} 
\usetikzlibrary{matrix,arrows,decorations.pathmorphing,patterns,positioning,calc}
\usepackage{enumitem}

\usepackage[bookmarksnumbered = true,pdfstartview=FitH]{hyperref}

\usepackage{amsmath}
\usepackage{amssymb}
\usepackage{amsthm}

\theoremstyle{plain}
\newtheorem{theorem}{Theorem}[section]
\newtheorem{proposition}[theorem]{Proposition}
\newtheorem{lemma}[theorem]{Lemma}
\newtheorem{corollary}[theorem]{Corollary}

\newtheorem{remark}[theorem]{Remark}

\newtheorem{definition}[theorem]{Definition}
\newtheorem{example}[theorem]{Example}

\newcommand{\N}{\mathbf{N}}
\newcommand{\Z}{\mathbf{Z}}
\newcommand{\F}{\mathbf{F}}
\newcommand{\GL}{\mathrm{GL}}
\newcommand{\PGL}{\mathrm{PGL}}
\newcommand{\PO}{\mathrm{PO}}
\newcommand{\A}{\mathbf{A}}
\newcommand{\PP}{\mathbf{P}}
\newcommand{\G}{\mathbf{G}}
\newcommand{\Autgp}[1]{\mathbf{Aut}_{#1}}
\newcommand{\sm}[1]{#1^{\textrm{sm}}}

\renewcommand{\phi}{\varphi}
\newcommand{\m}{\mathfrak{m}}
\newcommand{\p}{\mathfrak{p}}

\newcommand{\fonction}[4]{\left\{\begin{array}[c]{ccc}#1 & \longrightarrow & #2 \\ #3 & \longmapsto & #4 \end{array}\right.}

\DeclareMathOperator{\id}{id}
\DeclareMathOperator{\Spec}{Spec}
\DeclareMathOperator{\Aut}{Aut}
\DeclareMathOperator{\Hom}{Hom}

\DeclareMathOperator{\pr}{pr}
\DeclareMathOperator{\Pic}{Pic}

\DeclareMathOperator{\Gal}{Gal}

\DeclareMathOperator{\Vect}{Vect}

\DeclareMathOperator{\coker}{coker}

\title{Almost homogeneous curves over an arbitrary field}
\author{Bruno Laurent \thanks{Univ. Grenoble Alpes, CNRS, Institut Fourier, F-38000 Grenoble, France \newline \url{Bruno.Laurent@univ-grenoble-alpes.fr}}}
\date{}

\begin{document}

\maketitle

\begin{abstract}
We classify the pairs $(C,G)$ where $C$ is a seminormal curve over an arbitrary field $k$ and $G$ is a smooth connected algebraic group acting faithfully on $C$ with a dense orbit, and we determine the equivariant Picard group of $C$. We also give a partial classification when $C$ is no longer assumed to be seminormal.
\end{abstract}


\section{Introduction}

A variety which is homogeneous under the action of an algebraic group is a very symmetric object. The study of equivariant compactifications of homogeneous varieties leads to the notion of almost homogeneous varieties. These are the varieties with a dense orbit. For example, toric varieties are the normal almost homogeneous varieties under the action of a torus.

\bigskip
In case of curves, the situation is rather nice for several reasons. First, a curve is almost homogeneous under the action of a smooth connected algebraic group if and only if the action is non-trivial. Moreover the complement of the dense orbit consists only of finitely many fixed points.

Second, we have natural compactifications. More precisely, if $C$ is a regular curve over a field $k$ then there exists a regular projective curve $\widehat{C}$ and an open immersion $C \hookrightarrow \widehat{C}$. Such a curve satisfies a universal property: for any proper scheme $Y$ over $k$, every morphism $C \to Y$ extends uniquely to a morphism $\widehat{C} \to Y$. In particular $\widehat{C}$ is unique up to unique isomorphism. We call it the regular completion of $C$ and the points of $\widehat{C} \setminus C$ are called the points at infinity.

Finally, if $C$ is a projective curve then the functor which associates with a $k$-scheme $S$ the abstract group $\Aut_{S}(C \times S)$ of $S$-automorphisms of $C \times S$ is representable by an algebraic group denoted by $\Autgp{C}$ (see \cite[Ex. 7.1.2]{BRI_structure}).

\bigskip
The study of automorphisms of curves has a long story. A first step was a theorem of Adolf Hurwitz stating that a compact Riemann surface of genus at least $2$ has finitely many automorphisms. This result was generalized several times, and in particular Maxwell Rosenlicht gave in \cite[Th. p.4]{ROSFF} a version for smooth projective curves over an arbitrary field, in the language of algebraic function fields in one variable. He also classified in \cite[Th. p.10]{ROSFF} the algebraic function fields which he called ``exceptional'', that is, those of genus at least $1$ having infinitely many automorphisms (fixing a given place if the genus equals $1$). The link with regular curves is standard: there is an anti-equivalence between the category of regular curves over $k$ and the category of algebraic function fields in one variable over $k$, which associates with a curve $C$ its function field $k(C)$, and the closed points of $C$ correspond to the places of $k(C)$. Some years after, Peter Russell gave in \cite[Th. 4.2]{RUS} an interpretation of Rosenlicht's classification in geometric terms, which we state for simplicity in case $k$ is separably closed: a regular projective curve $C$ of genus at least $1$ has infinitely many automorphisms (fixing a given closed point if the genus equals $1$) if and only if $C$ is the regular completion of a torsor under a non-trivial form of the additive group $\G_{a,k}$.

\bigskip
From the geometric point of view, a natural approach to classify almost homogeneous curves is to understand the regular completions of homogeneous curves, and the behavior of the points at infinity. Vladimir Popov thus obtained in \cite[Ch. 7]{POP} a full classification of almost homogeneous curves over an algebraically closed field of characteristic zero, and determined their abstract automorphism group. Our objective is to extend this result by classifying the pairs $(C,G)$ where $C$ is a curve over an arbitrary field and $G$ is a smooth connected algebraic group acting faithfully on $C$ with a dense orbit. 

\bigskip
The complexity of the classification depends on the class of singularity of the curves. After regular curves, the next class to look at is the class of seminormal curves. Roughly speaking, these are the curves whose branches intersect as transversally as possible. They can be easily described in the language of pinchings developed by Daniel Ferrand in \cite{FER} (see Section \ref{section_normalization} and Lemma \ref{lemma_SN_curve_pinching}). We obtain the following theorem.

\begin{theorem}\label{th_general_classification}
Let $C$ be a seminormal curve, $G$ is a smooth connected algebraic group and $\alpha : G \times C \to C$ an action. The action is faithful and $C$ is almost homogeneous if and only if one of the following cases holds:  
\begin{enumerate}
\item (homogeneous curves)\label{case_homogeneous}
  \begin{enumerate}
  \item $C$ is a smooth projective conic and $G \simeq \Autgp{C}$;
  \item $C \simeq \A^1_k$ and $G \simeq \G_{a,k} \rtimes \G_{m,k}$ (acting by affine transformations);
  \item $G$ is a form of $\G_{a,k}$ and $C$ is a $G$-torsor;
  \item $G$ is a form of $\G_{m,k}$ and $C$ is a $G$-torsor;
  \item $C$ is a smooth projective curve of genus $1$ and $G \simeq \Autgp{C}^\circ$;
  \end{enumerate}
\item (regular, non-homogeneous curves)\label{case_regular_qhomogeneous}
  \begin{enumerate}
  \item $C \simeq \PP^1_k$ and $G \simeq \G_{a,k} \rtimes \G_{m,k}$;
  \item $G$ is a form of $\G_{a,k}$ and $C$ is the regular completion of a $G$-torsor;
  \item $C \simeq \A^1_k$ or $\PP^1_k$ and $G \simeq \G_{m,k}$;
  \item $C$ is a smooth projective conic and $G$ is the centralizer of a separable point \mbox{of degree $2$};
  \end{enumerate} 
\item (seminormal, singular, non-homogeneous curves)\label{case_SN_qhomogeneous}
  \begin{enumerate}
  \item\label{case_qhomogeneous_Ga} $G$ is a non-trivial form of $\G_{a,k}$ and $C$ is obtained by pinching the point at infinity $\widetilde{P}$ of the regular completion of a $G$-torsor on a point $P$ whose residue field $\kappa(P)$ is a strict subextension of $\kappa(\widetilde{P})/k$;
  \item $C$ is obtained by pinching two $k$-rational points of $\PP^1_k$ on a $k$-rational point and $G \simeq \G_{m,k}$;
  \item $C$ is obtained by pinching a separable point $\widetilde{P}$ of degree $2$ of a smooth projective conic $\widetilde{C}$ on a $k$-rational point, and $G$ is the centralizer of $\widetilde{P}$ in $\Autgp{\widetilde{C}}$.
  \end{enumerate}
\end{enumerate} 
\end{theorem}

Let us notice that, since in the case \ref{case_qhomogeneous_Ga} the curves are parameterized by subextensions of $\kappa(\widetilde{P})/k$, these curves can form an uncoutable family (see Remark \ref{rem_continuous_family}).

\bigskip
When $C$ is not necessarily seminormal, the situation is more complex and requires to study representations of algebraic groups. We give a classification in Sections \ref{section_arbitrary_Ga} and \ref{section_arbitrary_Gm}, except for almost homogeneous curves under the action of $\G_{a,k} \rtimes \G_{m,k}$ or a form of $\G_{a,k}$ when the field $k$ has positive characteristic, because in this case the representations of the forms of $\G_{a,k}$ (and even $\G_{a,k}$ itself) are not so well-understood. The full classification remains an open problem.

Another open problem is to determine the automorphism group scheme of almost homogeneous projective curves, especially for the curves obtained by pinching the regular completion of a torsor under a non-trivial form of $\G_{a,k}$.

\bigskip
For projective almost homogeneous curves under the action of an algebraic group $G$, one may want to try to embed them in the projectivization of a $G$-module. The standard tool is the notion of $G$-linearized line bundles. We describe in Proposition \ref{prop_pinching_linearized_line_bundles} the $G$-linearized line bundles over a variety obtained by a pinching. We use this description to determine in Theorem \ref{th_general_PicG} the equivariant Picard groups of the curves appearing in Theorem \ref{th_general_classification}.

\paragraph*{Notations and conventions.} We fix a base field $k$ and an algebraic closure $\overline{k}$, and denote by $k_s$ the separable closure of $k$ in $\overline{k}$. For all $k$-schemes $X$ and $S$, the base change $X \times_k S$ shall be denoted by $X_S$; in case $S= \Spec K$ for some field extension $K/k$, we simply write $X_K$. 

All morphisms between $k$-schemes are morphisms over $k$. A variety over $k$ is a separated scheme of finite type over $\Spec k$ which is geometrically integral. A curve is a variety of dimension $1$. An algebraic group over $k$ is a group scheme of finite type over $\Spec k$. A subgroup of an algebraic group is a (closed) subgroup scheme. We may consider non-smooth groups, but all the groups acting on a variety shall be assumed smooth.

For a smooth algebraic group $G$, we denote by $X(G) = \Hom_{k_s-\mathrm{gp}}(G_{k_s},\G_{m,k_s})$ the abstract group of characters of $G_{k_s}$. Let $\underline{X(G)}_{k_s}$ be the corresponding constant group scheme over $k_s$ and $\widehat{G}$ the \'etale sheaf given by $\widehat{G}(S) = \left(\Hom_{k_s-\textrm{sch}}(S_{k_s},\underline{X(G)}_{k_s})\right)^{\Gal(k_s/k)}$ for every $k$-scheme $S$. For any Galois extension $K/k$ we have $\widehat{G}(\Spec K) \simeq \Hom_{K-\mathrm{gp}}(G_K,\G_{m,K})$ and this group is simply denoted by $\widehat{G}(K)$.

\paragraph*{Acknowledgements.} I gratefully thank Michel Brion for his helpful and enlightening ideas and suggestions. I also thank Rapha\"el Achet, Mohamed Benzerga and Delphine Pol for fruitful discussions.


\section{Regular almost homogeneous curves}

\subsection{Scheme-theoretic actions}

In this section we recall some elementary facts about the actions of algebraic groups, in the setting of schemes. We use \cite{DEM} as a general reference.

\bigskip
Let $X$ be a variety, $G$ a smooth algebraic group with neutral element $e$ and an action $\alpha : \fonction{G \times X}{X}{(g,x)}{gx}$. For any $k$-rational point $x \in X(k)$ the orbit morphism $\alpha_x : \fonction{G}{X}{g}{gx}$ is flat and factorizes as $G \to G x \to X$ where $G x$ is a reduced scheme, the first morphism is faithfully flat and the second one is an immersion. Thus $G x$ is a $G$-stable subscheme of $X$, called the orbit of $x$ (see \cite[II 5.3.1]{DEM}).

For any closed subscheme $Y$ of $X$, the functors $N_G(Y)$ and $C_G(Y)$ which associate with each $k$-scheme $S$ the abstract groups 
\begin{align*}
N_G(Y)(S) & = \{g \in G(S) \mid  g \textrm{ induces an automorphism of } Y_S\} \\ &  = \{g \in G(S) \mid \forall S' \to S, \forall y \in Y(S') \subseteq X(S'), gy \in Y(S')\}
\end{align*}
and
\begin{align*}
C_G(Y)(S) & = \{g \in G(S) \mid  g \textrm{ induces the identity on } Y_S\} \\ & = \{g \in G(S) \mid \forall S' \to S, \forall y \in Y(S'), gy=y\}
\end{align*}
are representable by subgroups of $G$, called respectively the normalizer and the centralizer of $Y$ (see \cite[II 1.3.6]{DEM}). In case $Y=X$, the centralizer is called the kernel of the action and is a normal subgroup of $G$. In case $Y$ is a $k$-rational point $x$, the centralizer is denoted by $G_x$ and is called the isotropy subgroup of $x$; it is also the fiber at $x$ of $\alpha_x$, and $\alpha_x$ induces an isomorphism $G/G_x \simeq G x$ (see below for a reminder on quotients). 

Similarly, the functor $X^G$ which associates with each $k$-scheme $S$ the set 
\begin{align*}
X^G(S) & = \{x \in X(S) \mid \textrm{the orbit morphism } \alpha_x : G_S \to X_S \textrm{ is trivial}\} \\
& = \{x \in X(S) \mid \forall S' \to S, \forall g \in G(S'), gx=x \}
\end{align*}
is representable by a closed subscheme of $X$, called the subscheme of fixed points. Moreover $X^G(\overline{k})$ is the set of the elements of $X(\overline{k})$ which are fixed under the action of $G(\overline{k})$.

\bigskip

One would like to define $X$ to be homogeneous (resp. almost homogeneous) if there exists a unique orbit (respectively a dense orbit). Since $X$ may not have any $k$-rational point, a more convenient definition can be given as follows.

\begin{definition}\label{def_almost_hom}
The variety $X$ is said to be homogeneous (resp. almost homogeneous) under the action of $G$ if the morphism $\gamma=(\alpha,\pr_2) : \fonction{G \times X}{X \times X}{(g,x)}{(gx,x)}$ is surjective (resp. dominant).
\end{definition}

However, one recovers the natural definition in terms of orbits after a suitable field extension, as shown by the following lemmas.

\begin{lemma}
Let $K/k$ be a field extension. The following assertions are equivalent:
\begin{enumerate}[label=\roman*)]
 \item The variety $X$ is homogeneous under the action of $G$.
 \item The variety $X_K$ is homogeneous under the action of $G_K$.
 \item The action of the abstract group $G(\overline{k})$ on the set $X(\overline{k})$ is transitive.
 \item If there exists a $K$-rational point $x \in X(K)$ then the orbit $G_K x$ equals $X_K$.
\end{enumerate}
\end{lemma}

\begin{proof}
The assertions $i)$, $ii)$ and $iii)$ are equivalent because a morphism $f : Y \to Y'$ between schemes of finite type over $k$ is surjective if and only if $f_K$ is surjective, and also if and only if the induced map of sets $ : Y(\overline{k}) \to Y'(\overline{k})$ is surjective (see \cite[Exp. XII, Prop. 3.2]{SGA1}). In particular this holds for the morphism $\gamma$. Moreover for any $K$-rational point $x$ of $X_K$ the equality of schemes $G_K x = X_K$ is equivalent to the surjectivity of the orbit morphism $G_K \to X_K$, to the surjectivity of the induced map $G(\overline{K}) \to X(\overline{K})$ and to the transitivity of the action of $G(\overline{K})$ on $X(\overline{K})$.
\end{proof}

\begin{lemma}\label{lemma_almost_hom_field_extension}
The following assertions are equivalent:
\begin{enumerate}[label=\roman*)]
 \item The variety $X$ is almost homogeneous under the action of $G$.
 \item For every field extension $K/k$, the variety $X_K$ is almost homogeneous under the action of $G_K$.
 \item There exists a field extension $K/k$ such that the variety $X_K$ is almost homogeneous under the action of $G_K$.
 \item \label{unique_open_orbit}There exists an open subscheme $U$ of $X$ which is $G$-stable and homogeneous.
 \item \label{item_qh_orbit_point}There exists a field extension $K/k$ and a $K$-rational point $x \in X(K)$ such that the orbit $G_K x$ is an open subscheme of $X_K$.
\end{enumerate}
In \ref{item_qh_orbit_point} the extension $K/k$ can be chosen to be finite and separable. Furthermore, the open subscheme $U$ is unique, and we call it the open orbit.
\end{lemma}

\begin{proof}
Let $\pi : X_K \times X_K \to X \times X$ be the projection, which is an open morphism. The set-theoretic image of the morphism $\gamma_K : G_K \times X_K \to X_K \times X_K$ is $\pi^{-1}(\gamma(G \times X))$. So $\gamma_K(G_K \times X_K)$ contains a nonempty open subset of $X_K \times X_K$ if and only if $\gamma(G \times X)$ contains a nonempty open subset of $X \times X$. Thus it follows from Chevalley's theorem that $\gamma$ is dominant if and only if $\gamma_K$ is dominant (see \cite[I 3.3.8]{DEM}).

\medskip
Assume that $X$ is almost homogeneous. By \cite[Sect. 5, p.519]{DEMcre}, there exists a finite separable extension $K/k$ and a point $x \in X(K)$ such that the orbit $G_K x$ is open in $X_K$. Any element $\gamma$ of $\Gal(K/k)$ induces an automorphism of $X_K$. Then the open subscheme $\gamma(G_K x)$ is the orbit of $\gamma(x)$. The two open subschemes $G_K x$ and $\gamma(G_K x)$ must have a nonempty intersection, so these two orbits are equal. In other words, $G_K x$ is $\Gal(K/k)$-stable. Thus, by Galois descent, there exists an open subscheme $U$ of $X$ such that $U_K = G_K x$, and this subscheme is $G$-stable and homogeneous. Moreover, if $U'$ is a $G$-stable and homogeneous open subscheme of $X$ then $U_{\overline{k}}$ and $U'_{\overline{k}}$ are two open orbits in $X_{\overline{k}}$ so again we have $U_{\overline{k}} = U'_{\overline{k}}$ and $U=U'$.

\medskip
If $X$ contains a $G$-stable and homogeneous open subscheme $U$ then the image of $\gamma$ contains $U \times U$ so $\gamma$ is dominant.

\medskip
If \ref{item_qh_orbit_point} is true then the orbit $G_K x$ is a $G_K$-stable and homogeneous open subscheme of $X_K$ so $X_K$ is almost homogeneous.
\end{proof}

\begin{lemma}
Let \label{lemma_action_trivial_dim0}$G$ be a smooth connected algebraic group and $Z$ a zero-dimensional reduced scheme of finite type over $k$. The only action $\alpha : G \times Z \to Z$ is the trivial one.
\end{lemma}

\begin{proof}
We can assume that $k$ is separably closed. The scheme $Z$ is noetherian and zero-dimensional so it is affine and its underlying topological space is a finite set endowed with the discrete topology. For each $z \in Z$ (considered as an open subscheme of $Z$), $G \times z$ is irreducible so $\alpha(G \times z)$ is irreducible too. This set contains $ez=z$ so $\alpha(G \times z) = z$. Thus we have a factorization $\alpha : G \times z \to z$ and we can assume that $Z$ is integral. We may write $Z = \Spec K$ for some field extension $K/k$. Since $Z$ is of finite type over $k$, $K/k$ is a finite extension, and hence it is purely inseparable.

Since $G$ is geometrically reduced, $G(k)$ is dense in $G$. Then it suffices to show that the elements of $G(k)$ act trivially on $Z$. For $g \in G(k)$, the automorphism of $Z$ induced by $g$ corresponds to a $k$-automorphism of $K$. But $K/k$ is purely inseparable, so the unique $k$-automorphism of $K$ is the identity.
\end{proof}

\begin{remark}
This result is not true in general if $G$ is not smooth. Indeed, let $k$ be an imperfect field of characteristic $p$, $a \in k$ which is not a $p$th power, $K = k\left(a^{1/p}\right) = k[X]/(X^p-a)$ and $Z = \Spec K$. Then $Z$ is a zero-dimensional reduced scheme of finite type over $k$. However the infinitesimal group $\mu_p$ acts non trivially on $Z$ by multiplication on $X$.

The result is not true either in general if $Z$ is not reduced.
\end{remark}

\begin{lemma}
Let \label{lemma_existence_open_orbit}$G$ be a smooth connected algebraic group, $C$ a curve and $\alpha : G \times C \to C$ an action. The curve $C$ is almost homogeneous if and only if the action $\alpha$ is non-trivial. In this case, the open orbit is the complement of the subscheme $Z = C^G$ of fixed points.
\end{lemma}

\begin{proof}
If $C$ is not almost homogeneous then for $x \in C(\overline{k})$, the closure of the orbit $G_{\overline{k}}x$ in $C_{\overline{k}}$ is irreducible and cannot have dimension $1$, so the orbit is trivial. Thus $G$ acts trivially.

Assume that $C$ is almost homogeneous. Let $U = C \setminus Z$. Then $U_{\overline{k}}$ is the open orbit in $C_{\overline{k}}$. Indeed, let $x \in C(\overline{k})$ such that the orbit $G_{\overline{k}} x$ is open. By Lemma \ref{lemma_action_trivial_dim0}, the complement of this orbit is the set of fixed points of $C(\overline{k})$. So we have $U_{\overline{k}} = G_{\overline{k}} \setminus Z_{\overline{k}} = G_{\overline{k}} \setminus Z(\overline{k}) =  G_{\overline{k}} x$. 
\end{proof}

\bigskip
We now show how to restrict to faithful actions. First recall that if $H$ is a subgroup of a smooth algebraic group $G$ then by \cite[Exp. VIA, 3.2]{SGA3-1} there exists a smooth scheme of finite type $G/H$, an action $G \times G/H \to G/H$ and a $G$-equivariant morphism $\pi : G \to G/H$ which is a $H$-torsor (where $H$ acts on the right on $G$ by multiplication); in particular $\pi$ is faithfully flat. If $H$ is a normal subgroup then there exists a unique algebraic group structure on $G/H$ for which $\pi$ is a group morphism. For every $k$-scheme $S$, the exact sequence of algebraic groups $$1 \to H \to G \to G/H \to 1$$ yields an exact sequence of abstract groups $$1 \to H(S) \to G(S) \to (G/H)(S).$$ The last morphism need not be surjective, but it is surjective if $S= \Spec \overline{k}$. More generally, for any $k$-scheme $S$ and $\overline{g} \in (G/H)(S)$ there exists a morphism $S' \to S$ which is faithfully flat of finite presentation and an element $g \in G(S')$ such that $\overline{g} = \pi(g)$ in $(G/H)(S')$.

\begin{lemma}
Let $H$ be the kernel of the action $\alpha$. Then there exists a unique action $\beta : G/H \times X \to X$ such that the diagram \begin{tikzpicture}[baseline=(m.center)]
\matrix(m)[matrix of math nodes,
row sep=2.5em, column sep=2.5em,
text height=1.5ex, text depth=0.25ex,ampersand replacement=\&]
{G\times X \& X \\
G/H \times X \& \\};
\path[->] (m-1-1) edge node[above] {$\alpha$} (m-1-2);
\path[->](m-1-1) edge node[left] {$\pi\times\id_X$} (m-2-1);
\path[->] (m-2-1) edge node[below right] {$\beta$} (m-1-2);
\end{tikzpicture} commutes. Moreover the action $\beta$ is faithful, and $X$ is homogeneous (resp. almost homogeneous) under the action of $G$ if and only if it is so under the action of $G/H$.
\end{lemma}

\begin{proof}
The argument is very standard. The morphism $\pi \times \id_X : G \times X \to G/H \times X$ is a $H$-torsor so it is a categorical quotient. Since $\alpha : G \times X \to X$ is $H$-invariant (where $H$ acts on the right on the first factor of $G \times X$), there exists a unique morphism $\beta$ such that the diagram commutes.

\medskip
Let us show that $\beta$ is an action. We have to show that for every $k$-scheme $S$ and every $\overline{g}_1 \in (G/H)(S)$, $\overline{g}_2 \in (G/H)(S)$ and $x \in X(S)$ we have $\overline{g}_1(\overline{g}_2x) = (\overline{g}_1\overline{g}_2)x$ and $\pi(e) x = x$ (where $e$ is the neutral element of $G(S)$). The second equality is obvious. Let $S' \to S$ be a morphism which is faithfully flat of finite presentation and $g_1 \in G(S')$, $g_2 \in G(S')$ such that $\overline{g}_1 = \pi(g_1)$ and $\overline{g}_2 = \pi(g_2)$ in $(G/H)(S')$. Then in $X(S')$ we have $\overline{g}_1(\overline{g}_2x) = g_1(g_2x) = (g_1g_2)x = (\overline{g}_1\overline{g}_2)x$. Since the map $X(S) \to X(S')$ is injective (because $S' \to S$ is an epimorphism), the equality $\overline{g}_1(\overline{g}_2x) = (\overline{g}_1\overline{g}_2)x$ already holds in $X(S)$.

\medskip
Let $N$ be the kernel of the action $\beta$. We need to prove that for every $k$-scheme $S$, $N(S)$ is the trivial group. Let $\overline{g} \in N(S)$. Let $S' \to S$ and $g \in (G/H)(S')$ as previously. Then for every $S'' \to S'$ and $x \in X(S'')$ we have $gx = \overline{g}x = x$. So $g$ is in the kernel of the abstract action $G(S') \times X(S') \to X(S')$, that is, $g \in H(S')$. Then $\overline{g} = \pi(g) = e$ in $(G/H)(S')$ so, as before, $\overline{g} = e$ in $(G/H)(S)$.

Finally, since $\pi \times \id_X$ is surjective, the morphism $\gamma : G \times X \to X \times X$ is surjective (resp. dominant) if and only if so is the analogous morphism $G/H \times X \to X \times X$.
\end{proof}

\bigskip
We can also restrict to the case $G$ is connected.

\begin{lemma}\label{lemma_connected_homogeneous}
Let $G^\circ$ be the neutral component of $G$. Then $X$ is homogeneous (resp. almost homogeneous) under the action of $G$ if and only if it is so under the action of $G^\circ$.
\end{lemma}

\begin{proof}
We can assume that $k$ is algebraically closed. Assume first that $X$ is homogeneous under the action of $G$. The cosets $G^\circ g$ for $g \in G(k)$ form an open covering of $G$. Pick $x \in X(k)$. The orbit morphism $\alpha_x$ is open because it is flat. Hence the sets $\alpha_x(G^\circ g)$ form an open covering of $X$. Since $X$ is irreducible, the open subsets $\alpha_x(G^\circ)$ and $\alpha_x(G^\circ g)$ must have a common $k$-rational point. But $\alpha_x(G^\circ)(k)$ and $\alpha_x(G^\circ g)(k)$ are the orbits of $x$ and $gx$ under the abstract group $G^\circ(k)$. Then they are equal and the abstract group $G^\circ(k)$ acts transitively on $X(k)$.

\medskip
Assume now that $X$ is almost homogeneous under the action of $G$. By hypothesis, the morphism $\gamma : G \times X \to X \times X$ defined in Definition \ref{def_almost_hom} is dominant. Let us show that its restriction $\gamma_0$ to $G^\circ \times X$ is also dominant. For every irreducible component $G_i$ of $G$, the right multiplication by some $g_i \in G_i(k)$ induces isomorphisms $G^\circ \simeq G_i$ and $\overline{\gamma(G^\circ \times X)} \simeq \overline{\gamma(G_i \times X)}$. Moreover we have $\displaystyle \bigcup_i \overline{\gamma(G_i \times X)} = \overline{\bigcup_i \gamma(G_i \times X)} = \overline{\gamma(G\times X)} = X \times X$ (the first equality holds because there are finitely many irreducible components). Hence $\dim \overline{\gamma(G^\circ \times X)} = \dim X \times X$. But $\overline{\gamma(G^\circ \times X)}$ and $X \times X$ are irreducible, so they are equal and $\gamma_0$ is dominant.

The converses are trivial.
\end{proof}


\subsection{Forms of the multiplicative and additive groups}

Tori are very well-known. For the one-dimensional ones, we can specialize the general results of \cite[8.11 p.117]{BOREL} to get the following lemma.

\begin{lemma}
Let $G$ be a non-trivial form of $\G_{m,k}$. There exists a Galois extension $K/k$ of degree $2$ such that $G_K \simeq \G_{m,K}$. In addition, every character of $G$ is trivial. 
\end{lemma}

\begin{proof}
Let $K'/k$ be a finite Galois extension such that $G_{K'} \simeq \G_{m,K'}$. The Galois group $\Gamma' = \Gal(K'/k)$ acts on the group of characters $\widehat{G}(K') = \Hom_{K'-\mathrm{gp}}(G_{K'},\G_{m,K'})$, and an element of $\widehat{G}(K')$ is defined over $k$ if and only if it is $\Gamma'$-invariant. But $G$ is not isomorphic to $\G_{m,k}$ so not every character of $G_{K'}$ is defined over $k$. Thus $\Gamma'$ acts non-trivially on $\widehat{G}(K') \simeq \widehat{\G_{m,K'}}(K') \simeq \Z$. Then the kernel $\Gamma$ of the morphism $\Gamma' \to \Aut \Z \simeq \{-1,1\}$ is a normal subgroup of index $2$ of $\Gamma'$. By Galois correspondence, there exists a Galois subextension $K/k$ such that $[K:k]=2$, $\Gal(K'/K) = \Gamma$ and $\Gal(K/k) \simeq \Gamma'/\Gamma$. On the one hand every character of $G_{K'}$ is defined over $K$, so $G_{K} \simeq \G_{m,K}$. On the other hand, the group $\Gamma'/\Gamma \simeq \{-1,1\}$ acts non-trivially on $\widehat{G}(K) \simeq \Z$ so the only $(\Gamma'/\Gamma)$-invariant element of $\widehat{G}(K)$ is the trivial character.
\end{proof}

There is also a geometric description of the forms of $\G_{m,k}$ and their torsors in terms on conics. We first need an easy lemma about the group structures of $\A^1_k \setminus \{0\}$.

\begin{lemma}
Up to the choice of the neutral element, the only algebraic group structure on $\A^1_k\setminus\{0\}$ is $\G_{m,k}$.
\end{lemma}

\begin{proof}
Let $G$ be an algebraic group whose underlying space is $\A^1_k \setminus\{0\}$. After a suitable translation, we may assume that the neutral element is $1$. We can also assume that $k$ is algebraically closed, since the multiplication map $G \times G \to G$ and the usual multiplication on $\G_{m,k}$ are equal if and only if they become equal after some field extension. Let $x \in G(k)$. The left multiplication by $x$ is an automorphism of $\A^1_k \setminus\{0\}$, which corresponds to an algebra automorphism of $k[T,T^{-1}]$. Such an automorphism is given by $T \mapsto aT^\varepsilon$ for some $a \in k^\times$ and $\varepsilon \in \{-1,1\}$. The left multiplication by $x$ maps the neutral element $1$ to $x$, so $a=x$. If $\varepsilon = -1$ and $x \neq 1$ then there are fixed points (namely the square roots of $x$), which is impossible. Thus $\varepsilon = 1$ and as expected the law is given by $(x,y) \mapsto xy$.
\end{proof}

\begin{lemma}
Let \label{lemma_conic}$G$ be an algebraic group and $C$ a curve. Then $G$ is a non-trivial form of $\G_{m,k}$ and $C$ is a $G$-torsor if and only if there exists a smooth projective conic $\widehat{C}$, a Galois extension $K/k$ of degree $2$ and a point $P \in \widehat{C}$  (endowed with its structure of a reduced closed subscheme) with residue field $K$ such that $C = \widehat{C} \setminus \{P\}$ and $G$ is the centralizer of $P$ in $\Autgp{\widehat{C}}$.
\end{lemma}

\begin{proof}
Let $G$ be a non-trivial form of $\G_{m,k}$, $C$ a $G$-torsor and $\widehat{C}$ the regular completion of $C$. Let $K/k$ be a Galois extension of degree $2$ such that $G_K \simeq \G_{m,K}$. By Galois cohomology and Hilbert's theorem $90$, the set of isomorphism classes of $\G_{m,K}$-torsors over $K$ is in bijection with $H^1(\Gal(K_s/K),K_s^\times) = \{1\}$ so $C_K\simeq \A^1_K \setminus \{0\}$. Since the extension $K/k$ is separable, the curve $(\widehat{C})_K$ is regular so it is the regular completion of $C_K$. Thus $(\widehat{C})_K \simeq \PP^1_K$ and the complement of $C_K$ consists of two $K$-rational points. Therefore $\widehat{C}$ is a smooth projective conic and $\widehat{C}\setminus C$ consists of $1$ or $2$ closed points. The residue field $\kappa(P)$ of a point $P$ in $\widehat{C}\setminus C$ is a subextension of $K/k$, and the number of points above $P$ in $(\widehat{C})_K$ is equal to the degree of $\kappa(P)/k$ since this extension is separable. Hence $\widehat{C}\setminus C$ is either a set of two $k$-rational points or a unique point $P$ with $\kappa(P)=K$. In the first case we have $\widehat{C} \simeq \PP^1_k$, the underlying space of $G$ is $\A^1_k \setminus \{0\}$ and so $G \simeq \G_{m,k}$, which contradicts the hypothesis. Hence the second case holds and $G$ is the centralizer of $P$ in $\Autgp{\widehat{C}}$ since $\G_{m,K}$ is the centralizer of two $K$-rational points in $\Autgp{\PP^1_K}$. 

\medskip
Conversely, such a $G$ acts on the complement $C$ of $P$. The curve $C$ is a $G$-torsor because it becomes so after extension to $K$. As before we have $G_K \simeq \G_{m,K}$. Moreover $G$ is not isomorphic to $\G_{m,k}$, otherwise by the same argument as above we would have $C \simeq \A^1 \setminus \{0\}$ so the complement would be a set of two $k$-rational points.
\end{proof}

The forms of $\G_{a,k}$ were classified by Russell in \cite{RUS}. He obtained the following result.

\begin{proposition}
If \label{prop_Russell} $C$ (resp. $G$) is a form of $\A^1_k$ (resp. $\G_{a,k}$) then there exists a smallest finite field extension $K/k$ such that $C_K \simeq \A^1_K$ (resp. $G_K \simeq \G_{a,K}$) and this extension is purely inseparable. 
If $k$ has characteristic $p >0$ then the forms of $\G_{a,k}$ are the algebraic groups isomorphic to a subgroup of $\G_{a,k}^2$ given by an equation of the form $$y^{p^n} = x + a_1x^p + \cdots + a_mx^{p^m}$$ for some integers $m \geq 0$ and $n \geq 0$ and some coefficients $a_i$ in $k$. If $G$ is given by such an equation then the smallest field extension $K/k$ such that $G_K \simeq \G_{a,K}$ is $K= k(a_1^{p^{-n}},\ldots,a_m^{p^{-n}})$. Furthermore, the $G$-torsors are the schemes isomorphic to a closed subscheme $C$ of $\A^2_k$ given by an equation of the form $$y^{p^n} = b + x + a_1x^p + \cdots + a_mx^{p^m}$$ for some coefficient $b$ in $k$. The complement of $C$ in its regular completion $\widehat{C}$ is a non-smooth point $P$ whose residue field is a subextension of $K/k$.
\end{proposition}

\begin{remark}
For a form $G$ of $\G_{a,k}$, the equation above is not unique. Russell gave a formula to decide whether two equations give isomorphic groups.
\end{remark}


\subsection{Classification of regular almost homogeneous curves}

The key result for the classification is the well-known following proposition. 

\begin{proposition}\cite[Ex. 7.1.2 and 7.1.3]{BRI_structure}
Let \label{prop_aut_genus}$C$ be a smooth projective curve of genus $g$, $Q$ a closed point of $C$ (endowed with its structure of a reduced closed subscheme) and $\Autgp{C,Q}$ the centralizer of $Q$ in $\Autgp{C}$. If $g > 1$ (resp. $g=1$) then $\Autgp{C}$ (resp. $\Autgp{C,Q}$) is a finite \'etale group.
\end{proposition}

The following lemma is a special case of a result of Andr\'e Weil on birational actions. For the sake of completeness, we give a proof based on ideas of Rosenlicht (see \cite[Th. 15]{ROS56}). Note that Michel Demazure gave in \cite[Prop. 5]{DEMcre} a similar statement for smooth curves. 

\begin{lemma}\label{lemma_action_regular_completion}
Let $C$ be a regular curve, $\widehat{C}$ its regular completion, $G$ a smooth connected algebraic group and $\alpha : G \times C \to C$ an action. The action $\alpha$ lifts uniquely to an action on $\widehat{C}$. Moreover, $G$ acts trivially on the complement $Z = \widehat{C} \setminus C$ endowed with its structure of a reduced closed subscheme of $\widehat{C}$.
\end{lemma}

\begin{proof}
Let $\widehat{\alpha} : G \times \widehat{C} \dashrightarrow \widehat{C}$ be the rational action defined by $\alpha$. Let $U$ be its domain of definition. It is an open subscheme of $G \times \widehat{C}$ containing $G \times C$. In order to prove that $\widehat{\alpha}$ is defined everywhere, it is enough to show that for any $(g,x) \in G(\overline{k}) \times \widehat{C}(\overline{k})$ the element $gx$ is defined, that is, $(g,x) \in U(\overline{k})$. The scheme $G \times \widehat{C}$ is normal (since $G$ is smooth and $\widehat{C}$ is normal) and $\widehat{C}$ is proper over $k$ so $(G \times \widehat{C}) \setminus U$ has codimension at least $2$ in $G \times \widehat{C}$. Moreover the curve $(\widehat{C})_{k_s}$ is the regular completion of $C_{k_s}$ and $(G_{k_s} \times \widehat{C}_{k_s})\setminus U_{k_s}$ has codimension at least $2$ in $G_{k_s} \times \widehat{C}_{k_s}$. Thus we can assume $k=k_s$.

The subscheme $Z$ of $\widehat{C}$ consists of finitely many closed points $P_1,\ldots,P_r$ whose residue fields are purely inseparable extensions of $k$. We write $Z = \displaystyle \bigsqcup_{i=1}^r \Spec \kappa(P_i)$. The irreducible components $Z_i = G \times \Spec \kappa(P_i)$ of $G \times Z$ have codimension $1$ in $G \times \widehat{C}$. So $U \cap Z_i$ is a nonempty open subset of $Z_i$ (otherwise $Z_i \subseteq (G \times \widehat{C}) \setminus U$ and $Z_i$ would have codimension at least $2$). By \cite[IV2, Prop. 2.4.2]{EGA} the first projection $\pi_i : Z_i \to G$ is a homeomorphism, so $V = \displaystyle \bigcap_{i=1}^r \pi_i(U \cap Z_i)$ is a nonempty open subset of $G$ such that $V \times Z \subseteq U$. For $g \in V(\overline{k})$ and $x \in Z(\overline{k})$ we have $gx \in Z(\overline{k})$. Indeed if $gx \in C(\overline{k})$ then $g^{-1}(gx)$ is defined, so we have $g^{-1}(gx) = (g^{-1}g)x = ex = x$ (and all these elements are defined), which is impossible because $g^{-1}(gx) \in C(\overline{k})$ and $x \in Z(\overline{k})$. Finally $G(\overline{k}) = V(\overline{k})V(\overline{k})$, hence for any $g \in G(\overline{k})$ and $x \in Z(\overline{k})$, the element $gx$ is defined.

Let $N_G(Z)$ be the normalizer of $Z$ in $G$. Any element $g \in G(k)$ yields an automorphism of $\widehat{C}$, again denoted by $g$, such that the subset $Z$ is stable. Since $Z$ is a reduced subscheme, $g$ restricts to an automorphism of $Z$. Thus we have $N_G(Z)(k)=G(k)$. But $G$ is geometrically reduced so $G(k)$ is dense in $G$. So $N_G(Z)=G$, that is, the action on $\widehat{C}$ restricts to an action on $Z$. By Lemma \ref{lemma_action_trivial_dim0}, this action is trivial.
\end{proof}

\begin{remark}
In particular, the action of $G$ on $\widehat{C}$ gives a morphism $G \to \Autgp{\widehat{C}}$ which factorizes as $G \to \Autgp{\widehat{C},Z}$, where $\Autgp{\widehat{C},Z}$ is the centralizer of $Z$ in $\Autgp{\widehat{C}}$.
\end{remark}

From the above results we easily deduce a list of possible homogeneous curves (which can be found for example in \cite[Prop. 7.1.2]{POP} in case $k$ is algebraically closed).

\begin{lemma}
Let \label{lemma_possible_homogeneous}$C$ be a smooth curve. If $C$ is homogeneous under the action of a smooth algebraic group then $C$ is a form of $\PP^1_k$,  $\A^1_k$ or $\A^1_k \setminus \{0\}$, or $C$ is a projective curve of genus $1$.
\end{lemma}

Let us remark that the converse is false. More specifically, some forms of $\A^1_k$ are not homogeneous, as shown by a counterexample mentioned without proof by Russell in \cite{RUS}.

\begin{lemma}\label{lemma_form_A1_not_almost_hom}
We assume that $k$ is imperfect. Let $P$ be a point on $\PP^1_k$ such that the extension $\kappa(P)/k$ is purely inseparable of degree at least $3$. We endow $P$ with its structure of a reduced closed subscheme of $\PP^1_k$. Then $C = \PP^1_k \setminus \{P\}$ is a form of $\A^1_k$ such that the centralizer $\Autgp{\PP^1_k,P}$ of $P$ in $\Autgp{\PP^1_k}$ is infinitesimal.
\end{lemma}

\begin{proof}
Since $\kappa(P)/k$ is purely inseparable, there exists a unique point $P_{k_s}$ of $\PP^1_{k_s}$ above $P$ and we have $[\kappa(P_{k_s}):k_s] = [\kappa(P) : k]$. Thus we can assume $k=k_s$. Similarly, there exists a unique point $P_{\overline{k}}$ of $\PP^1_{\overline{k}}$ above $P$ and it is $\overline{k}$-rational. So $C_{\overline{k}} = \PP^1_{\overline{k}} \setminus \{P_{\overline{k}}\} \simeq \A^1_{\overline{k}}$, thus $C$ is a form of $\A^1_k$.

An element $f$ of $\Autgp{\PP^1_k,P}(\overline{k})$ is a homography fixing the closed subscheme $P_{\overline{k}}$ of $\PP^1_{\overline{k}}$, which has length $[\kappa(P):k] \geq 3$, thus $f$ is the identity. Hence $\Autgp{\PP^1_k,P}(\overline{k})$ is the trivial group.
\end{proof}

We can now prove the case \ref{case_homogeneous} of Theorem \ref{th_general_classification}.

\begin{theorem}
Let \label{th_classification_homogeneous}$C$ be a curve and $G$ a smooth connected algebraic group acting faithfully on $C$. The curve $C$ is homogeneous under the action of $G$ if and only if one of the following cases holds: 
\begin{enumerate}[label=(\alph*)]
  \item $C$ is a smooth projective conic and $G \simeq \Autgp{C}$;
  \item \label{case_class_hom_2}$C \simeq \A^1_k$ and $G \simeq \G_{a,k} \rtimes \G_{m,k}$ (acting by affine transformations);
  \item $G$ is a form of $\G_{a,k}$ and $C$ is a $G$-torsor;
  \item $G$ is a form of $\G_{m,k}$ and $C$ is a $G$-torsor;
  \item $C$ is a smooth projective curve of genus $1$ and $G \simeq \Autgp{C}^\circ$.
\end{enumerate}
\end{theorem}

\begin{proof}
If $C$ is homogeneous under the action of $G$ then $C$ is smooth so $C_{\overline{k}}$ is smooth too and we can consider its regular completion $\widehat{C_{\overline{k}}}$. By Lemmas \ref{lemma_lift_action} and \ref{lemma_existence_open_orbit}, the action of $G_{\overline{k}}$ on $C_{\overline{k}}$ lifts to an action on $\widehat{C_{\overline{k}}}$, and $C_{\overline{k}}$ is the complement of the subscheme of fixed points $Z = (\widehat{C_{\overline{k}}})^{G_{\overline{k}}}$.

\medskip
By Lemma \ref{lemma_possible_homogeneous}, either $\widehat{C_{\overline{k}}}$ has genus $1$ or $\widehat{C_{\overline{k}}}$ is isomorphic to $\PP^1_{\overline{k}}$. In case $\widehat{C_{\overline{k}}}$ has genus $1$, it is a torsor under $\Autgp{\widehat{C_{\overline{k}}}}^\circ$ (after the choice of a $\overline{k}$-rational point, $\widehat{C_{\overline{k}}}$ is an elliptic curve $E$ and we have $\Autgp{E}^\circ = E$ where $E$ acts on itself by translation). Since $G_{\overline{k}}$ is a subgroup of $\Autgp{\widehat{C_{\overline{k}}}}^\circ$ and $\dim \Autgp{\widehat{C_{\overline{k}}}}^\circ = 1$, we have $G_{\overline{k}} = \Autgp{\widehat{C_{\overline{k}}}}^\circ$ and $C_{\overline{k}} = \widehat{C_{\overline{k}}}$ . 

\medskip
We now assume $\widehat{C_{\overline{k}}} = \PP^1_{\overline{k}}$. Then $G_{\overline{k}}$ is a subgroup of $\Autgp{\PP^1_{\overline{k}}} = \PGL_{2,\overline{k}}$ and $Z$ consists of $0$, $1$ or $2$ points. If $G_{\overline{k}}$ is a strict subgroup of $\PGL_{2,\overline{k}}$ then, by \cite[Cor. 11.6]{BOREL}, $G_{\overline{k}}$ is solvable so (up to conjugation) it is a subgroup of the standard Borel subgroup $\G_{a,\overline{k}}\rtimes \G_{m,\overline{k}}$. Thus if $Z$ is empty then $C_{\overline{k}} \simeq \PP^1_{\overline{k}}$ and $G_{\overline{k}} = \PGL_{2,\overline{k}}$. If $Z$ consists of $2$ points then $C_{\overline{k}} \simeq \A^1_{\overline{k}} \setminus \{0\}$ and $G_{\overline{k}} \simeq \G_{m,\overline{k}}$. If $Z$ consists of $1$ point then $C_{\overline{k}} \simeq \A^1_{\overline{k}}$ and $G_{\overline{k}} \simeq \G_{a,\overline{k}}$ or $G_{\overline{k}} \simeq \G_{a,\overline{k}} \rtimes \G_{m,\overline{k}}$.

\medskip
The classification over $k$ follows immediately, except in case \ref{case_class_hom_2}. We assume $C_{\overline{k}} \simeq \A^1_{\overline{k}}$ and $G_{\overline{k}} \simeq \G_{a,\overline{k}} \rtimes \G_{m,\overline{k}}$. Let us prove $C \simeq \A^1_k$. By Proposition \ref{prop_Russell}, we can assume that $k$ is separably closed. Let $L/k$ be a (purely inseparable extension) such that $C_L \simeq \A^1_L$. As above, we have $G_L \simeq \G_{a,L} \rtimes \G_{m,L}$. Let $T$ be a maximal torus of $G$. Then $T_L$ is a maximal torus of $G_L$ so $T_L \simeq \G_{m,L}$. Since the extension $L/k$ is purely inseparable, we already have $T \simeq \G_{m,k}$. The curve $C$ is geometrically reduced so the subset $C(k)$ is dense in $C$. If for every $y \in C(k)$ the orbit morphism $T \to C$ of $y$ were constant then we would have $C^T(k) = C(k)$, so the subscheme $C^T$ of fixed points would be equal to $C$, contradicting the faithfulness of the action. Let $y \in C(y)$ be such that the orbit morphism $T \to C$ is non-constant. Let $\widehat{C}$ be the regular completion of $C$. Then the orbit morphism extends to a non-constant morphism $\PP^1_k \to \widehat{C}$. Thus $\widehat{C}$ is a regular projective curve of genus $0$ having $k$-rational points, so $\widehat{C} \simeq \PP^1_k$. Hence $C$ is a strict open subscheme of $\PP^1_k$, containing strictly $T \simeq \G_{m,k}$. So $C \simeq \A^1_k$, and thus $G \simeq \G_{a,k} \rtimes \G_{m,k}$.
\end{proof}

\begin{remark}
We use the term ``conic'' as a synonym of ``projective curve of arithmetic genus $0$''. Indeed, if $C$ is a projective curve of arithmetic genus $0$ then it is smooth, and the anticanonical bundle $\omega_C^{\otimes -1}$ is very ample and yields a closed immersion $C \hookrightarrow \PP^2_k$ such that $C$ is given by a homogeneous equation of degree $2$. Let $q$ be the corresponding quadratic form. Then we have $\Autgp{C} \simeq \PO(q)$ and it is a form of $\PGL_{2,k}$.
\end{remark}

The case \ref{case_regular_qhomogeneous} of Theorem \ref{th_general_classification} is a direct consequence.

\begin{corollary}
Let \label{cor_classification_regular_almost_homogeneous}$C$ be a regular curve, $G$ is a smooth connected algebraic group and $\alpha : G \times C \to C$ an action. The action is faithful and $C$ is almost homogeneous if and only if one of the following cases holds: 
\begin{enumerate}[label=(\alph*)]
  \item $C \simeq \PP^1_k$ and $G \simeq \G_{a,k} \rtimes \G_{m,k}$;
  \item $G$ is a form of $\G_{a,k}$ and $C$ is the regular completion of a $G$-torsor;
  \item $C \simeq \A^1_k$ or $\PP^1_k$ and $G \simeq \G_{m,k}$;
  \item \label{case_2d}$C$ is a smooth projective conic and $G$ is the centralizer of a separable point of degree $2$.
\end{enumerate}
\end{corollary}

\begin{proof}
By Lemma \ref{lemma_almost_hom_field_extension}, if $C$ is almost homogeneous under the action of $G$ then it contains a $G$-stable open subscheme $U$ which is a homogeneous curve. Moreover $C$ is contained in the regular completion of $U$. For the case \ref{case_2d}, we use Lemma \ref{lemma_conic}.
\end{proof}

\section{Seminormal almost homogeneous curves}\label{sect_SN_curves}

\subsection{Normalization and pinching}\label{section_normalization}

We shall deduce the classification of seminormal almost homogeneous curves from the case of regular curves. In order to do so, we must link the action of a group $G$ on a curve with the action on the normalized curve.

We first need to recover a variety $X$ from its normalization $\widetilde{X}$. In case the field $k$ is algebraically closed and the singular locus of $X$ is a finite set, Jean-Pierre Serre gave in \cite[Ch. IV]{SerreGrpAlg} an explicit description by constructing the underlying space of $X$ and its structure sheaf. This can also be done in the more general language of ``pinchings''.  

The conductor $\mathcal{C}$ of the normalization $\nu : \widetilde{X} \to X$ is the coherent sheaf of ideals defined as the annihilator $$\mathcal{C}=\mathrm{Ann}_{\mathcal{O}_X} \nu_*\mathcal{O}_{\widetilde{X}}/\mathcal{O}_X.$$ For any $x \in X$, the stalk $\left(\nu_*\mathcal{O}_{\widetilde{X}}\right)_x$ is identified with the integral closure $\overline{\mathcal{O}}_{X,x}$ of $\mathcal{O}_{X,x}$ and we have \begin{align*} \mathcal{C}_x & = \mathrm{Ann}_{\mathcal{O}_{X,x}} \left(\nu_*\mathcal{O}_{\widetilde{X}}\right)_x/\mathcal{O}_{X,x} \\
& = \{ a \in \mathcal{O}_{X,x} \mid a \left(\nu_*\mathcal{O}_{\widetilde{X}}\right)_x \subseteq \mathcal{O}_{X,x} \}.
\end{align*} In other words, $\mathcal{C}_x$ is the largest ideal of $\mathcal{O}_{X,x}$ which is an ideal of $\overline{\mathcal{O}}_{X,x}$, that is to say, the conductor ideal in the classical sense. 

Let $i : Z \to X$ the closed immersion defined by $\mathcal{C}$ and $\widetilde{Z}$ the scheme-theoretic inverse image of $Z$ by $\nu$ in $\widetilde{X}$, given by the cartesian square 
\begin{tikzpicture}[baseline=(m.center)]
\matrix(m)[matrix of math nodes,
row sep=2.5em, column sep=2.5em,
text height=1.5ex, text depth=0.25ex,ampersand replacement=\&]
{\widetilde{Z} \& \widetilde{X} \\
Z \& X\\};
\path[->] (m-1-1) edge node[above] {$j$}(m-1-2);
\path[->] (m-1-1) edge node[left] {$\lambda$} (m-2-1);
\path[->] (m-2-1) edge node[below] {$i$} (m-2-2);
\path[->] (m-1-2) edge node[right] {$\nu$} (m-2-2);
\end{tikzpicture} called the conductor square. By base change, $j$ is a closed immersion and $\lambda$ is finite and surjective.

\begin{lemma}\label{lemma_conductor_base_change}
\begin{enumerate}
 \item The underlying space of $Z$ is the set of non-normal points of $X$. 
 \item The morphism $\nu$ induces an isomorphism from $\widetilde{X}\setminus\widetilde{Z}$ to $X \setminus Z$. 
 \item The morphism $\lambda : \widetilde{Z} \to Z$ is finite and schematically dominant.
 \item The square is cocartesian in the category of locally ringed spaces.
 \item For every separable extension $K/k$, 
 \begin{tikzpicture}[baseline=(m.center)]
\matrix(m)[matrix of math nodes,
row sep=2.5em, column sep=2.5em,
text height=1.5ex, text depth=0.25ex,ampersand replacement=\&]
{\widetilde{Z}_K \& \widetilde{X}_K \\
Z_K \& X_K\\};
\path[->] (m-1-1) edge node[above] {$j_K$}(m-1-2);
\path[->] (m-1-1) edge node[left] {$\lambda_K$} (m-2-1);
\path[->] (m-2-1) edge node[below] {$i_K$} (m-2-2);
\path[->] (m-1-2) edge node[right] {$\nu_K$} (m-2-2);
\end{tikzpicture}
is the conductor square of $X_K$.
\end{enumerate}
\end{lemma}

\begin{proof}
The underlying space of $Z$ is $\{x \in X \mid 1 \notin \mathcal{C}_x \} = \{x \in X \mid \mathcal{O}_{X,x} \subsetneq \overline{\mathcal{O}}_{X,x}\}$. Then $X \setminus Z$ is the normal locus of $X$ and $\nu$ induces an isomorphism from $\nu^{-1}(X \setminus Z) = \widetilde{X}\setminus\widetilde{Z}$ to $X \setminus Z$. Moreover $\lambda$ is finite and schematically dominant because so is $\nu$ and the square is cartesian.

The square is cocartesian if and only if for every open subscheme $U$ of $X$, the square 
\begin{tikzpicture}[baseline=(m.center)]
\matrix(m)[matrix of math nodes,
row sep=2.5em, column sep=2.5em,
text height=1.5ex, text depth=0.25ex,ampersand replacement=\&]
{\lambda^{-1}i^{-1}(U) \& \nu^{-1}(U) \\
i^{-1}(U) \& U\\};
\path[->] (m-1-1) edge node[above] {$j$}(m-1-2);
\path[->] (m-1-1) edge node[left] {$\lambda$} (m-2-1);
\path[->] (m-2-1) edge node[below] {$i$} (m-2-2);
\path[->] (m-1-2) edge node[right] {$\nu$} (m-2-2);
\end{tikzpicture}
is cocartesian. The morphism $\nu : \nu^{-1}(U) \to U$ is the normalization and $i^{-1}(U)$ is the closed subscheme of $U$ defined by the conductor of $\nu : \nu^{-1}(U) \to U$. Thus we can assume that $X$ is affine. Let us write $X = \Spec A$. We have $\widetilde{X} = \Spec \overline{A}$ where $\overline{A}$ is the integral closure of $A$. Let $\mathfrak{c} = \{a \in A \mid a \overline{A} \subseteq A \}$ be the conductor ideal. For any prime ideal $\p \in \Spec A$, the integral closure of the localization $A_\p$ is $(\overline{A})_\p$ and, since $\overline{A}$ is a $A$-module of finite type, $\mathfrak{c}_\p$ is the conductor ideal of $(\overline{A})_\p$ in $A_\p$. Thus $\mathcal{C}$ is the coherent sheaf of ideals associated with $\mathfrak{c}$ and we have $Z = \Spec A/\mathfrak{c}$ and $\widetilde{Z} = \Spec \overline{A}/\mathfrak{c}$. Consequently, by \cite[Lemma 1.3 and Th. 5.1]{FER} the square is cocartesian.

The variety $\widetilde{X}_K$ is normal and $\nu_K$ is the normalization. In order to prove that the diagram obtained after the field extension to $K$ is the conductor square, we can assume that $X$ is affine. Then the result is a direct consequence of \cite[I 2.10, Cor. 2 p.40]{BOUAC1}.
\end{proof}

Conversely Ferrand showed that under a mild assumption the variety $X$ is obtained by ``pinching'' $\widetilde{Z}$ onto $Z$. As a special case of \cite[Th. 5.4 and Prop. 5.6]{FER} we have the following result.

\begin{proposition}\label{prop_existence_pinching}
Let $\widetilde{X}$ be a $k$-scheme, $j : \widetilde{Z} \to \widetilde{X}$ a closed immersion and $\lambda : \widetilde{Z} \to Z$ a finite and schematically dominant morphism. If every finite set of points of $\widetilde{Z}$ is contained in an affine open subset of $\widetilde{X}$ then there exists a unique $k$-scheme $X$ such that we have a square 
\begin{tikzpicture}[baseline=(m.center)]
\matrix(m)[matrix of math nodes,
row sep=2.5em, column sep=2.5em,
text height=1.5ex, text depth=0.25ex,ampersand replacement=\&]
{\widetilde{Z} \& \widetilde{X} \\
Z \& X\\};
\path[->] (m-1-1) edge node[above] {$j$}(m-1-2);
\path[->] (m-1-1) edge node[left] {$\lambda$} (m-2-1);
\path[->] (m-2-1) edge node[below] {$i$} (m-2-2);
\path[->] (m-1-2) edge node[right] {$\nu$} (m-2-2);
\end{tikzpicture}
which is cocartesian in the category of locally ringed spaces. Furthermore this square is cartesian, $\nu$ is finite and schematically dominant, $i$ is a closed immersion and $\nu$ induces an isomorphism from $\widetilde{X}\setminus\widetilde{Z}$ to $X \setminus Z$. We call such a square a ``pinching diagram''.

If $\widetilde{X}$ is a variety (resp. a proper variety) then so is $X$. Finally, if $\widetilde{X}$ is normal then $\nu$ is the normalization. 
\end{proposition}

\begin{remark} 
\begin{enumerate}[wide, labelwidth=!, labelindent=0pt, label=\roman*)] 
 \item The last part is not explicitly stated in Ferrand's article. A detailed justification can be found in Sean Howe's master thesis \cite{HOW}. 
 \item If $\widetilde{X}$ is a curve then the condition that every finite set of points of $\widetilde{Z}$ is contained in an affine open subset of $\widetilde{X}$ is automatically satisfied. In this case, $\widetilde{Z}$ and $Z$ are finite schemes and $Z$ is determined by the injective morphism $\lambda^\sharp : \mathcal{O}(Z) \to \mathcal{O}(\widetilde{Z})$. This corresponds to the intuitive picture: a curve has finitely many singular points and to normalize this curve one ``separates'' the branches; conversely, the curve can be recovered from its normalization by gluing points together.
\end{enumerate}
\end{remark}

An action on the variety can be lifted to an action on the normalized variety.

\begin{lemma}\label{lemma_lift_action}
Let $G$ be a smooth algebraic group, $X$ a variety, $\alpha : G \times X \to X$ an action and 
\begin{tikzpicture}[baseline=(m.center)]
\matrix(m)[matrix of math nodes,
row sep=2.5em, column sep=2.5em,
text height=1.5ex, text depth=0.25ex,ampersand replacement=\&]
{\widetilde{Z} \& \widetilde{X} \\
Z \& X\\};
\path[->] (m-1-1) edge node[above] {$j$}(m-1-2);
\path[->] (m-1-1) edge node[left] {$\lambda$} (m-2-1);
\path[->] (m-2-1) edge node[below] {$i$} (m-2-2);
\path[->] (m-1-2) edge node[right] {$\nu$} (m-2-2);
\end{tikzpicture} the conductor square. There exists a unique action $\widetilde{\alpha} : G \times \widetilde{X} \to \widetilde{X}$ such that $\nu$ is equivariant. Moreover the square 
\begin{tikzpicture}[baseline=(m.center)]
\matrix(m)[matrix of math nodes,
row sep=2.5em, column sep=2.5em,
text height=1.5ex, text depth=0.25ex,ampersand replacement=\&]
{G\times \widetilde{X} \& \widetilde{X} \\
G \times X \& X \\};
\path[->] (m-1-1) edge node[above] {$\widetilde{\alpha}$} (m-1-2);
\path[->](m-1-1) edge node[left] {$\id\times\nu$} (m-2-1);
\path[->] (m-2-1) edge node[below] {$\alpha$} (m-2-2);
\path[->] (m-1-2) edge node[right] {$\nu$} (m-2-2);
\end{tikzpicture}
is cartesian and the actions $\alpha$ and $\widetilde{\alpha}$ have the same kernel. Finally, the closed subschemes $Z$ and $\widetilde{Z}$ are $G$-stable.
\end{lemma}

\begin{proof}
The existence of the action $\widetilde{\alpha}$ is given in \cite[Prop. 2.5.1]{BRI_structure}. The square is cartesian because $u : \fonction{G \times X}{G \times X}{(g,x)}{(g,gx)}$ is an isomorphism making the diagram
\begin{tikzpicture}[baseline=(m.center)]
\matrix(m)[matrix of math nodes,
row sep=2.5em, column sep=2.5em,
text height=1.5ex, text depth=0.25ex,ampersand replacement=\&]
{G \times X \& G\times X \\
\& X \\};
\path[->] (m-1-1) edge node[above] {$u$} (m-1-2);
\path[->] (m-1-1) edge node[below left] {$\alpha$} (m-2-2);
\path[->] (m-1-2) edge node[right] {$\pr_2$} (m-2-2);
\end{tikzpicture}
commute  (and similarly for $\widetilde{\alpha}$).

\medskip
The smooth locus $U$ of $X$ and $\nu^{-1}(U)$ are schematically dense open subschemes of $X$ and $\widetilde{X}$. Thus for any $k$-scheme $S$, the open subschemes $U_S$ and $\nu^{-1}(U)_S = (\nu \times \id_S)^{-1}(U_S)$ are schematically dense in $X_S$ and $\widetilde{X}_S$. They are $G_S$-stable and the morphism $\nu \times \id_S : \nu^{-1}(U)_S \to U_S$ is a $G_S$-equivariant isomorphism. Hence an element $g \in G(S)$ induces the identity on $X_S$ if and only if it induces the identity on $U_S$, if and only if it induces the identity on $\nu^{-1}(U)_S$, if and only if it induces the identity on $\widetilde{X}_S$. Consequently $\alpha$ and $\widetilde{\alpha}$ have the same kernel.

\medskip
Let $\mathcal{C}$ be the conductor sheaf of $\nu$. By \cite[IV2, Prop. 6.8.5]{EGA}, the scheme $G \times \widetilde{X}$ is normal. Then $\id \times \nu : G \times \widetilde{X} \to G \times X$ is the normalization. It follows from \cite[I.2.10, Cor. 2 p.40]{BOUAC1} that $\alpha^*\mathcal{C}$ and $\pr_2^*\mathcal{C}$ are both equal to the conductor sheaf of $\id \times \nu$. In particular they are equal as subsheaves of $\mathcal{O}_{G\times X}$ (using the canonical isomorphisms $\alpha^* \mathcal{O}_X \simeq \mathcal{O}_{G \times X} \simeq \pr_2^* \mathcal{O}_X$). Therefore we have $\alpha^{-1}(Z) = G \times Z$ as subschemes of $G \times X$, thus $Z$ is $G$-stable. Since the morphism $\nu$ is equivariant, $\widetilde{Z}$ is $G$-stable too.
\end{proof}

\begin{lemma}
Let \label{lemma_descent_action}
\begin{tikzpicture}[baseline=(m.center)]
\matrix(m)[matrix of math nodes,
row sep=2.5em, column sep=2.5em,
text height=1.5ex, text depth=0.25ex,ampersand replacement=\&]
{\widetilde{Z} \& \widetilde{X} \\
Z \& X\\};
\path[->] (m-1-1) edge node[above] {$j$}(m-1-2);
\path[->] (m-1-1) edge node[left] {$\lambda$} (m-2-1);
\path[->] (m-2-1) edge node[below] {$i$} (m-2-2);
\path[->] (m-1-2) edge node[right] {$\nu$} (m-2-2);
\end{tikzpicture} 
be a pinching diagram as in Proposition \ref{prop_existence_pinching}. Let $G$ be an algebraic group and $\widetilde{\alpha} : G \times \widetilde{X} \to \widetilde{X}$ and $\beta : G \times Z \to Z$ two actions. If the closed subscheme $\widetilde{Z}$ is $G$-stable and $\lambda$ is equivariant then there exists a unique action $\alpha : G \times X \to X$ such that $Z$ is $G$-stable under this action, the induced action on $Z$ is $\beta$ and $\nu$ is equivariant.
\end{lemma}

\begin{proof}
It follows from \cite[Th. 3.11]{HOW} that the square 
\begin{tikzpicture}[baseline=(m.center)]
\matrix(m)[matrix of math nodes,
row sep=2.5em, column sep=2.5em,
text height=1.5ex, text depth=0.25ex,ampersand replacement=\&]
{G \times \widetilde{Z} \& G \times \widetilde{X} \\
G \times Z \& G \times X\\};
\path[->] (m-1-1) edge node[above] {$\id \times j$}(m-1-2);
\path[->] (m-1-1) edge node[left] {$\id \times \lambda$} (m-2-1);
\path[->] (m-2-1) edge node[below] {$\id \times i$} (m-2-2);
\path[->] (m-1-2) edge node[right] {$\id \times \nu$} (m-2-2);
\end{tikzpicture} 
is cocartesian in the category of locally ringed spaces. Let $\widetilde{\beta} : G \times \widetilde{Z} \to \widetilde{Z}$ be the action induced by $\widetilde{\alpha}$. By assumption the diagram 
\begin{tikzpicture}[baseline=(m.center)]
\matrix(m)[matrix of math nodes,
row sep=1.25em, column sep=0.75em,
text height=1.5ex, text depth=0.25ex,ampersand replacement=\&]
{G\times \widetilde{Z} \& \& G\times \widetilde{X} \& \& \\
\& \& \& \& \widetilde{X} \\
G\times Z \& \& G\times X \& \& \\
\& \& \& \& \\
\& Z \& \& \& X \\};
\path[->] (m-1-1) edge node[above] {$\id \times j$}(m-1-3);
\path[->] (m-1-1) edge node[left] {$\id \times \lambda$} (m-3-1);
\path[->] (m-3-1) edge node[below] {$\id \times i$} (m-3-3);
\path[->] (m-1-3) edge node[right] {$\id \times \nu$} (m-3-3);
\path[->] (m-1-3) edge [bend left=15] node[above] {$\widetilde{\alpha}$} (m-2-5);
\path[->] (m-3-1) edge [bend right=15] node[below left] {$\beta$} (m-5-2);
\path[->] (m-5-2) edge node[below] {$i$} (m-5-5);
\path[->] (m-2-5) edge node[right] {$\nu$} (m-5-5);
\end{tikzpicture}
commutes. Then there exists a unique morphism $\alpha : G \times X \to X$ which completes it. It remains to prove that $\alpha$ is action, that is to say the two composite morphisms $G \times G \times X \xrightarrow{\id \times \alpha} G \times X \xrightarrow{\alpha} X$ and $G \times G \times X \xrightarrow{\mu \times \id} G \times X \xrightarrow{\alpha} X$ are equal, as well as $\Spec k \times X \xrightarrow{e \times \id} G \times X \xrightarrow{\alpha} X$ and the second projection. This can be done in both cases by using the fact that the square 
\begin{tikzpicture}[baseline=(m.center)]
\matrix(m)[matrix of math nodes,
row sep=2.5em, column sep=2.5em,
text height=1.5ex, text depth=0.25ex,ampersand replacement=\&]
{G \times G \times \widetilde{Z} \& G \times G \times \widetilde{X} \\
G \times G \times Z \& G \times G \times X\\};
\path[->] (m-1-1) edge (m-1-2);
\path[->] (m-1-1) edge (m-2-1);
\path[->] (m-2-1) edge (m-2-2);
\path[->] (m-1-2) edge (m-2-2);
\end{tikzpicture} 
is cocartesian and showing that the two morphisms complete the same diagram. The details are left to the reader.
\end{proof}

The normalization behaves well with respect to almost-homogeneity.

\begin{lemma}
Let \label{lemma_almost_homogeneous_lift}$G$ be a smooth connected algebraic group, $X$ a variety, 
\begin{tikzpicture}[baseline=(m.center)]
\matrix(m)[matrix of math nodes,
row sep=2.5em, column sep=2.5em,
text height=1.5ex, text depth=0.25ex,ampersand replacement=\&]
{\widetilde{Z} \& \widetilde{X} \\
Z \& X\\};
\path[->] (m-1-1) edge node[above] {$j$}(m-1-2);
\path[->] (m-1-1) edge node[left] {$\lambda$} (m-2-1);
\path[->] (m-2-1) edge node[below] {$i$} (m-2-2);
\path[->] (m-1-2) edge node[right] {$\nu$} (m-2-2);
\end{tikzpicture} 
the conductor square, $\alpha : G \times X \to X$ an action and $\widetilde{\alpha} : G \times \widetilde{X} \to \widetilde{X}$ the action given by Lemma \ref{lemma_lift_action}. The variety $X$ is almost homogeneous if and only if $\widetilde{X}$ is almost homogeneous. In this case, if in addition $X$ is a curve $C$ then $Z$ and $\widetilde{Z}$ are contained in the subschemes $C^G$ and $\widetilde{C}^G$ of fixed points.
\end{lemma}

\begin{proof}
The morphisms $\gamma : G \times X \to X \times X$ and $\widetilde{\gamma} : G \times \widetilde{X} \to \widetilde{X} \times \widetilde{X}$ of Definition \ref{def_almost_hom} make the diagram 
\begin{tikzpicture}[baseline=(m.center)]
\matrix(m)[matrix of math nodes,
row sep=2.5em, column sep=2em,
text height=1.5ex, text depth=0.25ex,ampersand replacement=\&]
{G\times \widetilde{X} \& \widetilde{X} \times \widetilde{X}\\
G\times X \& X \times X\\};
\path[->] (m-1-1) edge node[above] {$\widetilde{\gamma}$} (m-1-2);
\path[->] (m-1-1) edge node[left] {$\id\times\nu$} (m-2-1);
\path[->] (m-2-1) edge node[below] {$\gamma$} (m-2-2);
\path[->] (m-1-2) edge node[right] {$\nu \times \nu$} (m-2-2);
\end{tikzpicture}
commute. The morphism $\nu$ is surjective and closed, so $\id \times \nu$ is surjective and $\nu \times \nu$ is surjective and closed. If $\widetilde{X}$ is almost homogeneous then $\widetilde{\gamma}$ is dominant, so $\gamma$ is dominant too and $X$ is almost homogeneous.

\medskip
Conversely, assume that $X$ is almost homogeneous. Let $U$ be the open orbit in $X$, given by Lemma \ref{lemma_almost_hom_field_extension}. Then $U$ is smooth so $\nu^{-1}(U) \simeq U$. Hence $\nu^{-1}(U)$ is a $G$-stable and homogeneous open subscheme of $\widetilde{X}$, so $\widetilde{X}$ is almost homogeneous.

\medskip
Assume that $X$ is an almost homogeneous curve $C$. By Lemma \ref{lemma_existence_open_orbit}, $U$ and $\nu^{-1}(U)$ are the complements of the subschemes of fixed points. Since $Z$ is the singular locus of $C$ and $U$ is smooth, $Z$ is contained in $C \setminus U$. Then $\widetilde{Z} = \nu^{-1}(Z)$ is contained in $\widetilde{C} \setminus \nu^{-1}(U)$.
\end{proof}


\subsection{Seminormality}

In this section we recall some properties of seminormal schemes. We use the article \cite{SWAN} of Richard Swan as a general reference.

\begin{definition}
Let $f : Y \to X$ be a morphism between reduced noetherian schemes. We say that $f$ is a pseudo-isomorphism if $f$ is integral and bijective and, for all $y \in Y$, the morphism $f^\sharp : \kappa(f(y)) \to \kappa(y)$ between the residue fields is an isomorphism.

We say that $X$ is seminormal if every pseudo-isomorphism $f : Y \to X$ is an isomorphism.
\end{definition}

\begin{remark}
\begin{enumerate}[wide, labelwidth=!, labelindent=0pt, label=\roman*)]
 \item In the literature, $f$ is often called a ``quasi-isomorphism'' or a ``subintegral morphism''. We propose the name ``pseudo-isomorphism'' to avoid the confusion with the other notions of quasi-isomorphism (such as the homological one). Moreover, the name ``subintegral morphism'' can be misleading since, contrary to what it can suggest, $f$ is indeed integral.
 \item By \cite[I, Prop. 3.5.8]{EGA} and \cite[IV2, Prop. 2.4.5]{EGA}, a pseudo-isomorphism is a universal homeomorphism.
 \item Being a pseudo-isomorphism is a local property on the target.
\end{enumerate}
\end{remark}

Normal schemes are obtained by gluing normal rings. Similarly, seminormal schemes are obtained from seminormal rings.

\begin{definition}\cite[Lemma 2.2]{SWAN}
Let $A$ be a reduced noetherian ring and set $$A^+ = \{ b \in \overline{A} \mid \forall \p \in \Spec A, b \in A_\p + R(\overline{A}_\p) \subseteq \overline{A}_\p \}$$ where $\overline{A}$ is the integral closure of $A$ in its total ring of fractions, $A_\p$ and $\overline{A}_\p$ are the localizations by the multiplicative set $A \setminus \p$ and $R(\overline{A}_\p)$ is the Jacobson radical of $\overline{A}_\p$. The morphism $\Spec A^+ \to \Spec A$ is a pseudo-isomorphism and $A^+$ is the largest subextension of $A \subseteq \overline{A}$ with this property. The ring $A$ is said to be seminormal if $A = A^+$.
\end{definition}

Being a seminormal ring is a local property.

\begin{lemma}\cite[Cor. 2.10]{SWAN}
Let $A$ be a reduced noetherian ring. The following assertions are equivalent: 
\begin{enumerate}[label=\roman*)]
 \item The ring $A$ is seminormal.
 \item For every $\p \in \Spec A$, $A_\p$ is seminormal.
 \item For every maximal ideal $\m \in \Spec A$, $A_\m$ is seminormal.
\end{enumerate}
\end{lemma}

As an analogous of the normalization, there is a seminormalization. Its construction is well-known and we recall it briefly.

\begin{lemma}
Let $X$ be a reduced noetherian scheme and $\nu : \widetilde{X} \to X$ the normalization. There exists a seminormal scheme $X^+$, a pseudo-isomorphism $\sigma : X^+ \to X$ and a factorization $\widetilde{X} \to X^+ \xrightarrow{\sigma} X$ of $\nu$, satisfying the following universal property:
\begin{center}
\parbox{0.9\textwidth}{
for every seminormal scheme $Y$ and every morphism $f : Y \to X$, there is a unique factorization $Y \to X^+ \xrightarrow{\sigma} X$ of $f$.
}
\end{center}
The morphism $\sigma$ is called the seminormalization.
\end{lemma}

\begin{proof}
If $X$ is affine then we write $X = \Spec A$ and $X^+ = \Spec A^+$. The morphism $\sigma : X^+ \to X$ and the factorization $\widetilde{X} \to X^+ \xrightarrow{\sigma} X$ of $\nu$ are given by the inclusions $A \subseteq A^+ \subseteq \overline{A}$. More generally, we cover $X$ by affine open subschemes $U_i = \Spec A_i$ and it follows from \cite[Prop. 2.9]{SWAN} that we can glue the $\Spec A_i^+$ together to get a seminormal scheme $X^+$, a pseudo-isomorphism $\sigma : X^+ \to X$ and a factorization $\widetilde{X} \to X^+ \xrightarrow{\sigma} X$ of $\nu$. 

The universal property is given by \cite[Th. 4.1]{SWAN} in case $X$ and $Y$ are affine. The general case follows readily by taking affine open coverings.
\end{proof}

\begin{corollary}
Let \label{cor_SN_is_local}$X$ be a reduced noetherian scheme. The following assertions are equivalent: 
\begin{enumerate}[label=\roman*)]
 \item The scheme $X$ is seminormal.
 \item For every affine open subscheme $U$ of $X$, the ring $\mathcal{O}_X(U)$ is seminormal.
 \item For every $x \in X$, the ring $\mathcal{O}_{X,x}$ is seminormal.
 \item Every open subscheme of $X$ is seminormal.
 \item There exists an open covering of $X$ by seminormal schemes.
\end{enumerate}
\end{corollary}

Swan proved in \cite[Th. 1]{SWAN} that if $X$ is an affine noetherian reduced scheme then the pullback morphism $\Pic X \to \Pic(X \times \A^1)$ is an isomorphism if and only if $X$ is seminormal, generalizing a result of Carlo Traverso (\cite[Th. 3.6]{Traverso}). For non-affine schemes, we have the following special case of \cite[Lemma 3.6]{BRI_lin}.

\begin{lemma}
Let \label{lemma_Pic_homotopy}$X$ be a separated seminormal scheme. The pullback morphism $\Pic X \to \Pic(X \times \A^1)$ is an isomorphism.
\end{lemma}

The seminormalization commutes with smooth base changes.

\begin{lemma}
Let \label{lemma_SN_base_change}$f : Y \to X$ be a smooth morphism between noetherian reduced schemes and $\sigma : X^+ \to X$ the seminormalization. 
Then $Y \times_X X^+$ is seminormal and the projection $Y \times_X X^+ \to Y$ is the seminormalization.
\end{lemma}

\begin{proof}
By Corollary \ref{cor_SN_is_local}, we can assume that $X$ is affine and write $X = \Spec A$. Let $V = \Spec C$ be an affine open subscheme of $Y$ contained in $f^{-1}(U)$. Then the result follows from \cite[Prop. 5.1]{GRETRA}.
\end{proof}

We now state an analogue of Lemma \ref{lemma_lift_action} (see \cite[Lemma 3.5]{BRI_lin}).

\begin{corollary}
Let \label{cor_SN_lift_action}$G$ be a smooth algebraic group, $X$ a variety, $\alpha : G \times X \to X$ an action and $\sigma : X^+ \to X$ the seminormalization. There exists a unique action $\alpha^+ : G \times X^+ \to X^+$ such that $\sigma$ is equivariant. Moreover the square 
\begin{tikzpicture}[baseline=(m.center)]
\matrix(m)[matrix of math nodes,
row sep=2.5em, column sep=2.5em,
text height=1.5ex, text depth=0.25ex,ampersand replacement=\&]
{G\times X^+ \& X^+ \\
G \times X \& X \\};
\path[->] (m-1-1) edge node[above] {$\alpha^+$} (m-1-2);
\path[->](m-1-1) edge node[left] {$\id\times\sigma$} (m-2-1);
\path[->] (m-2-1) edge node[below] {$\alpha$} (m-2-2);
\path[->] (m-1-2) edge node[right] {$\sigma$} (m-2-2);
\end{tikzpicture}
is cartesian and the actions $\alpha$ and $\alpha^+$ have the same kernel. 
\end{corollary}

\subsection{Classification of seminormal almost homogeneous curves}

We have the following characterization of seminormal curves, which is a consequence of the well-known characterization \cite[Cor. 2.7]{GRETRA} of seminormal Japanese rings satisfying Serre's condition $S_2$ (a reduced noetherian ring of Krull dimension $1$ satisfies this condition).

\begin{lemma}
Let $C$ be a curve and 
\begin{tikzpicture}[baseline=(m.center)]
\matrix(m)[matrix of math nodes,
row sep=2.5em, column sep=2.5em,
text height=1.5ex, text depth=0.25ex,ampersand replacement=\&]
{\widetilde{Z} \& \widetilde{C} \\
Z \& C\\};
\path[->] (m-1-1) edge node[above] {$j$}(m-1-2);
\path[->] (m-1-1) edge node[left] {$\lambda$} (m-2-1);
\path[->] (m-2-1) edge node[below] {$i$} (m-2-2);
\path[->] (m-1-2) edge node[right] {$\nu$} (m-2-2);
\end{tikzpicture}
the conductor square. The curve $C$ is seminormal if and only if $\widetilde{Z}$ is reduced.
\end{lemma}

\begin{proof}
Let $U$ be an affine open subscheme of $C$. Then 
\begin{tikzpicture}[baseline=(m.center)]
\matrix(m)[matrix of math nodes,
row sep=2.5em, column sep=2.5em,
text height=1.5ex, text depth=0.25ex,ampersand replacement=\&]
{\lambda^{-1}i^{-1}(U) \& \nu^{-1}(U) \\
i^{-1}(U) \& U\\};
\path[->] (m-1-1) edge node[above] {$j$}(m-1-2);
\path[->] (m-1-1) edge node[left] {$\lambda$} (m-2-1);
\path[->] (m-2-1) edge node[below] {$i$} (m-2-2);
\path[->] (m-1-2) edge node[right] {$\nu$} (m-2-2);
\end{tikzpicture}
is the conductor square. So by Corollary \ref{cor_SN_is_local}, we can assume than $C$ is affine. Then the result follows from \cite[Cor. 2.7]{GRETRA}.
\end{proof}

In this case $Z$ is reduced too. Therefore if a smooth connected algebraic group $G$ acts on $C$ then by Lemmas \ref{lemma_action_trivial_dim0} and \ref{lemma_lift_action}, $G$ acts trivially on $Z$ and $\widetilde{Z}$. The following lemma shows that the converse is true, even when we consider a pinching diagram which is not the conductor square.

\begin{lemma}
Let \label{lemma_SN_curve_pinching}$C$ be a curve obtained by a pinching diagram 
\begin{tikzpicture}[baseline=(m.center)]
\matrix(m)[matrix of math nodes,
row sep=2.5em, column sep=2.5em,
text height=1.5ex, text depth=0.25ex,ampersand replacement=\&]
{\widetilde{Z} \& \widetilde{C} \\
Z \& C\\};
\path[->] (m-1-1) edge node[above] {$j$}(m-1-2);
\path[->] (m-1-1) edge node[left] {$\lambda$} (m-2-1);
\path[->] (m-2-1) edge node[below] {$i$} (m-2-2);
\path[->] (m-1-2) edge node[right] {$\nu$} (m-2-2);
\end{tikzpicture} 
where $\widetilde{C}$ is a regular curve and $\widetilde{Z}$ is a reduced finite scheme. Then $C$ is seminormal. Moreover, if a smooth connected algebraic group $G$ acts on $\widetilde{C}$ so that $\widetilde{Z}$ is $G$-stable then the action on $\widetilde{C}$ descends to an action on $C$ such that $Z$ is $G$-stable. 
\end{lemma}

\begin{proof}
Let us show that $C$ is seminormal. The morphism $\nu$ induces an isomorphism from $\widetilde{C} \setminus \widetilde{Z}$ to $C \setminus Z$ so it suffices to show that for every point $P \in Z$, the ring $\mathcal{O}_{C,P}$ is seminormal. Since $\lambda$ is schematically dominant and $\widetilde{Z}$ is reduced, $Z$ is a reduced finite scheme. Let $U$ be an affine open subscheme of $C$ containing $P$ but no other point of $Z$. The square
\begin{tikzpicture}[baseline=(m.center)]
\matrix(m)[matrix of math nodes,
row sep=2.5em, column sep=2em,
text height=1.5ex, text depth=0.25ex,ampersand replacement=\&]
{\lambda^{-1}(\Spec \kappa(P)) \& \nu^{-1}(U)\\
\Spec \kappa(P) \& U\\};
\path[->] (m-1-1) edge node[above] {$j$} (m-1-2);
\path[->] (m-1-1) edge node[left] {$\lambda$} (m-2-1);
\path[->] (m-2-1) edge node[below] {$i$} (m-2-2);
\path[->] (m-1-2) edge node[right] {$\nu$} (m-2-2);
\end{tikzpicture}
is still a pinching diagram. Thus we can assume that $C$ is affine and $Z = \Spec \kappa(P)$.

We write $C = \Spec A$ and $\widetilde{C} = \Spec \overline{A}$ where $\overline{A}$ is the integral closure of $A$. The point $P$ corresponds to a maximal ideal $\m$ of $A$ and the points of $\widetilde{Z}$ to the maximal ideals $\overline{\m}_1,\ldots,\overline{\m}_r$ of $\overline{A}$ above $\m$. Then $Z = \Spec A / \m$ and $\widetilde{Z} = \displaystyle \coprod_{i=1}^r \Spec \overline{A}/\overline{\m}_i = \Spec \overline{A}/(\overline{\m}_1\cdots\overline{\m}_r)$. By \cite[Lemma 1.3]{FER}, we have $\overline{\m}_1\cdots\overline{\m}_r=\m$. In other words, the maximal ideal $\m A_\m$ of $\mathcal{O}_{C,P}$ is the Jacobson radical of its integral closure $\overline{\mathcal{O}_{C,P}}$. Since the only prime ideals of $\mathcal{O}_{C,P}$ are $(0)$ and $\m A_\m$ (because $\mathcal{O}_{C,P}$ is an integral local ring of Krull dimension $1$), it follows immediately from the definition of $\mathcal{O}_{C,P}^+$ that $\mathcal{O}_{C,P}$ is seminormal.

By Lemma \ref{lemma_action_trivial_dim0}, $G$ acts trivially on $Z$ and $\widetilde{Z}$. Hence by Lemma \ref{lemma_descent_action}, the action of $G$ on $\widetilde{C}$ descends to an action on $C$.
\end{proof}

A consequence of Lemma \ref{lemma_SN_curve_pinching} is that the property of being a seminormal almost homogeneous curve descends by separable field extensions. To see this, we first need the following particular case of \cite[Lemma C.4.1 p.649]{CGP}.

\begin{lemma}\label{lemma_largest_smooth_subgroup}
Let $G$ be an algebraic group. There exists a largest smooth subgroup $\sm{G}$ of $G$. Moreover, for every separable field extension $K/k$, we have $\sm{G}(K) = G(K)$ and $\sm{(G_K)} = (\sm{G})_K$.
\end{lemma}

\begin{proposition}\label{prop_SN_qhom_descent}
Let $C$ be a seminormal curve and $K/k$ a separable field extension. If $C_K$ is almost homogeneous under the action of a smooth connected algebraic group then so is $C$.
\end{proposition}

\begin{proof}
Let $G$ be a smooth connected algebraic group over $K$ acting on $C_K$ such that $C_K$ is almost homogeneous. We can assume that the action is faithful. Let 
\begin{tikzpicture}[baseline=(m.center)]
\matrix(m)[matrix of math nodes,
row sep=2.5em, column sep=2.5em,
text height=1.5ex, text depth=0.25ex,ampersand replacement=\&]
{\widetilde{Z} \& \widetilde{C} \\
Z \& C\\};
\path[->] (m-1-1) edge node[above] {$j$}(m-1-2);
\path[->] (m-1-1) edge node[left] {$\lambda$} (m-2-1);
\path[->] (m-2-1) edge node[below] {$i$} (m-2-2);
\path[->] (m-1-2) edge node[right] {$\nu$} (m-2-2);
\end{tikzpicture} 
be the conductor square, $\widehat{C}$ the regular completion of $\widetilde{C}$ and $Z' = (\widehat{C} \setminus \widetilde{C}) \sqcup \widetilde{Z}$ endowed with its structure of a reduced closed subscheme of $\widehat{C}$. Since the extension $K/k$ is separable, by Lemma \ref{lemma_conductor_base_change} the diagram
\begin{tikzpicture}[baseline=(m.center)]
\matrix(m)[matrix of math nodes,
row sep=2.5em, column sep=2.5em,
text height=1.5ex, text depth=0.25ex,ampersand replacement=\&]
{\widetilde{Z}_K \& \widetilde{C}_K\\
Z_K \& C_K \\};
\path[->] (m-1-1) edge (m-1-2);
\path[->] (m-1-1) edge (m-2-1);
\path[->] (m-2-1) edge (m-2-2);
\path[->] (m-1-2) edge (m-2-2);
\end{tikzpicture}
is still the conductor square. Moreover $\widehat{C}_K$ is the regular completion of $\widetilde{C}_K$ and $Z'_K = (\widehat{C}_K \setminus \widetilde{C}_K) \sqcup \widetilde{Z}_K$ is reduced. By Lemmas \ref{lemma_action_regular_completion} and \ref{lemma_lift_action}, the action of $G$ on $C_K$ lifts to an action on $\widehat{C}_K$ such that $Z'_K$ is fixed. In other words, up to isomorphism, $G$ is a subgroup of the centralizer $\Autgp{\widehat{C}_K,Z'_K}$ of $Z'_K$ in $\Autgp{\widehat{C}_K}$. In addition $G$ is smooth, so by Lemma \ref{lemma_largest_smooth_subgroup}, it is a subgroup of $\sm{(\Autgp{\widehat{C}_K,Z'_K})}$. Then $\widetilde{C}_K$ is almost homogeneous under the action of $\sm{(\Autgp{\widehat{C}_K,Z'_K})}$. We have $\sm{(\Autgp{\widehat{C}_K,Z'_K})} = (\sm{(\Autgp{\widehat{C},Z'})})_K$ so $\widetilde{C}$ is almost homogeneous under the action of $\sm{(\Autgp{\widehat{C},Z'})}$. By definition, this group acts trivially on $\widetilde{Z}$ so, by Lemma \ref{lemma_descent_action}, the action on $\widetilde{C}$ descends to an action on $C$ and $C$ is almost homogeneous. Finally, by Lemma \ref{lemma_connected_homogeneous}, $C$ is almost homogeneous under the action of the neutral component of $\sm{(\Autgp{\widehat{C},Z'})}$.
\end{proof}

\begin{remark}
\begin{enumerate}[wide, labelwidth=!, labelindent=0pt, label=\roman*)] 
 \item Proposition \ref{prop_SN_qhom_descent} does not extend to inseparable extensions, even for smooth curves. Indeed by Lemma \ref{lemma_form_A1_not_almost_hom} there exists a curve $C$ and a purely inseparable extension $K/k$ such that $C_K \simeq \A^1_K$ and every action of a smooth connected algebraic group on $C$ is trivial.
 \item With the notations of the proof, the largest smooth connected group acting faithfully on $C$ is the neutral component of $\sm{(\Autgp{\widehat{C},Z'})}$.
\end{enumerate}
\end{remark}

We deduce the classification of seminormal almost homogeneous curves, which is the case \ref{case_SN_qhomogeneous} of Theorem \ref{th_general_classification}.

\begin{theorem}
Let \label{th_classification_SN_almost_homogeneous}$C$ be a singular seminormal curve, $G$ is a smooth connected algebraic group and $\alpha : G \times C \to C$ an action. The action is faithful and $C$ is almost homogeneous if and only if one of the following cases holds: 
\begin{enumerate}[label=(\alph*)]
 \item $G$ is a non-trivial form of $\G_{a,k}$ and $C$ is obtained by pinching the point at infinity $\widetilde{P}$ of the regular completion of a $G$-torsor on a point $P$ whose residue field $\kappa(P)$ is a strict subextension of $\kappa(\widetilde{P})/k$;
 \item $C$ is obtained by pinching two $k$-rational points of $\PP^1_k$ on a $k$-rational point and $G \simeq \G_{m,k}$;
 \item $C$ is obtained by pinching a separable point $\widetilde{P}$ of degree $2$ of a smooth projective conic $\widetilde{C}$ on a $k$-rational point, and $G$ is the centralizer of $\widetilde{P}$ in $\Autgp{\widetilde{C}}$.
\end{enumerate} 
\end{theorem}

\begin{proof}
Let $\nu : \widetilde{C} \to C$ be the normalization. By Lemmas \ref{lemma_almost_homogeneous_lift} and \ref{lemma_SN_curve_pinching}, $C$ is almost homogeneous if and only if $\widetilde{C}$ is almost homogeneous, and $C$ is obtained by pinching a set of fixed points of $\widetilde{C}$ endowed with its structure of a reduced closed subscheme. The different possible pairs $(\widetilde{C},G)$ are given by Corollary \ref{cor_classification_regular_almost_homogeneous}. 

We cannot have $\widetilde{C} = \PP^1_k$ and $G = \G_{a,k} \rtimes \G_{m,k}$ or $G = \G_{a,k}$ because we would have to pinch the $k$-rational point $\infty$ onto a $k$-rational point, then $\nu$ would be a pseudo-isomorphism, so $\nu$ would be an isomorphism and $C$ would be regular. The other cases are similar.
\end{proof}

\begin{remark}\label{rem_continuous_family}
\begin{enumerate}[wide, labelwidth=!, labelindent=0pt, label=\roman*)]
\item Let $G$ be a non-trivial form of $\G_{a,k}$, $\widetilde{C}$ the regular completion of a $G$-torsor, $\widetilde{P}$ the point at infinity, $K_1$ and $K_2$ strict subextensions of $\kappa(\widetilde{P}) / k$, and $C_1$ and $C_2$ the curves obtained by the pinching diagrams 
\begin{tikzpicture}[baseline=(m.center)]
\matrix(m)[matrix of math nodes,
row sep=2.5em, column sep=2.5em,
text height=1.5ex, text depth=0.25ex,ampersand replacement=\&]
{\Spec \kappa(\widetilde{P}) \& \widetilde{C} \\
\Spec K_1 \& C_1\\};
\path[->] (m-1-1) edge node[above] {$j$}(m-1-2);
\path[->] (m-1-1) edge node[left] {$\lambda_1$} (m-2-1);
\path[->] (m-2-1) edge node[below] {$i_1$} (m-2-2);
\path[->] (m-1-2) edge node[right] {$\nu_1$} (m-2-2);
\end{tikzpicture}
and
\begin{tikzpicture}[baseline=(m.center)]
\matrix(m)[matrix of math nodes,
row sep=2.5em, column sep=2.5em,
text height=1.5ex, text depth=0.25ex,ampersand replacement=\&]
{\Spec \kappa(\widetilde{P}) \& \widetilde{C} \\
\Spec K_2 \& C_2\\};
\path[->] (m-1-1) edge node[above] {$j$}(m-1-2);
\path[->] (m-1-1) edge node[left] {$\lambda_2$} (m-2-1);
\path[->] (m-2-1) edge node[below] {$i_2$} (m-2-2);
\path[->] (m-1-2) edge node[right] {$\nu_2$} (m-2-2);
\end{tikzpicture}. 
These diagrams are the conductor squares. If there exists an isomorphism $\psi : C_1 \to C_2$ then, by the universal property of normalization, $\psi$ can be lifted to an automorphism $\widetilde{\psi}$ of $\widetilde{C}$. Then $\widetilde{\psi}$ induces an automorphism of $\kappa(\widetilde{P})$ mapping $K_1$ to $K_2$. But the extension $\kappa(\widetilde{P}) / k$ is purely inseparable, so the unique automorphism of $\kappa(\widetilde{P})$ is the identity and we must have $K_1 = K_2$. Therefore, distinct subextensions yield distinct curves. 
\item The family of seminormal curves which are almost homogeneous under the action of a non-trivial form of $\G_{a,k}$ can be very large, as shown by the following example.
\end{enumerate}
\end{remark}


\begin{example}
Set $p=2$, $k= \F_2(a,b)$ and consider the form $G$ of $\G_{a,k}$ given by the equation $y^4 = x + ax^2 + b^2x^4$ (see Proposition \ref{prop_Russell}). We claim that the residue field of the point at infinity of $G$ is $\F_2(a^{1/2},b^{1/2})$. For $c \in k$, $k(a^{1/2}+cb^{1/2})$ is a subextension of $\F_2(a^{1/2},b^{1/2})/k$, and different values of $c$ yield pairwise non-isomorphic subextensions. As a consequence, we have a family parameterized by $k$ of almost homogeneous curves under the action of $G$ which have the same normalization.

We now prove our claim on the residue field. The scheme-theoretic closure $\overline{G}$ of $G$ in $\PP^2_k$ has homogeneous equation $y^4 = xz^3+ax^2z^2+b^2x^4$ and the regular completion of $G$ is the normalization of $\overline{G}$. In the chart $(x=1)$, the equation reads $y^4 = z^3 + az^2 + b^2$. Set $A=\dfrac{k[y,z]}{(y^4 - z^3 - az^2 - b^2)}$ and $B=\dfrac{k[y,w]}{(y^2-w^3-aw-b)}$. Let us show that the normalization in this chart is given by the ring morphism $A \to B$ sending $z$ to $w^2-b$. First this morphism in injective and finite. The ring $B$ is an integral domain and has the same fraction field than $A$ because $w = \dfrac{y^2-a}{z}$. By the Jacobian criterion, every point of $\Spec B$ is smooth except the point given by the ideal $\m=(w^2-b)$ (which besides corresponds to the point at infinity of $G$). This point is nonetheless regular because $\m$ is principal. Hence $B$ is integrally closed and is the integral closure of $A$. Thus the residue field of the point at infinity is $\dfrac{k[y,w]/(y^2-w^3-aw-b)}{(w^2-b)} \simeq \F_2(a^{1/2},b^{1/2})$.
\end{example}

\section{Arbitrary curves}

Popov gave a classification of almost homogeneous curves over an algebraically closed field in \cite[Ch. 7]{POP} and used the language of Serre (\cite[Ch. IV]{SerreGrpAlg}) to describe them. We extend the result to an arbitrary field (except for almost homogeneous curves under the action of $\G_{a,k}$ or $\G_{a,k}\rtimes \G_{m,k}$) by using the language of pinchings.

\subsection{\texorpdfstring{Almost homogeneous curves under the action of $\G_{a,k}$ in characteristic zero}{Almost homogeneous curves under the action of Ga in characteristic zero}}\label{section_arbitrary_Ga}

We consider the projective coordinates $[t:u]$ on $\PP^1_k$. For $n \geq 0$, we denote by $\PP^1_{k,n}$ the curve defined by the pinching diagram 
\begin{center}
\begin{tikzpicture}[baseline=(m.center)]
\matrix(m)[matrix of math nodes,
row sep=3.5em, column sep=4.5em,
text height=3.5ex, text depth=2ex,ampersand replacement=\&]
{\Spec \dfrac{k[u]}{(u^{n})} \& \PP^1_k\\
\Spec k \& \PP^1_{k,n} \\};
\path[->] (m-1-1) edge node[above] {$j$} (m-1-2);
\path[->] (m-1-1) edge node[left] {$\lambda$} (m-2-1);
\path[->] (m-2-1) edge node[below] {$i$} (m-2-2);
\path[->] (m-1-2) edge node[right] {$\nu$} (m-2-2);
\end{tikzpicture}
\end{center}
where $\Spec \dfrac{k[u]}{(u^{n})}$ is the $n$th infinitesimal neighborhood of the point $\infty$. By Lemma \ref{lemma_descent_action} the natural actions of $\G_{a,k}$ and $\G_{a,k} \rtimes \G_{m,k}$ on $\PP^1_k$ descend to faithful actions on $\PP^1_{k,n}$.

\begin{theorem}
We \label{th_classification_Popov_Ga}assume that $k$ has characteristic zero. Let $C$ be a curve and $\alpha : \G_{a,k} \times C \to C$ an action. The action is faithful (and $C$ is almost homogeneous) if and only if $C \simeq \A^1_k$ or if there exists $n \geq 0$ such that $C \simeq \PP^1_{k,n}$, and the action is the natural one.
\end{theorem}

\begin{proof}
We adapt Popov's proof of \cite[Th. 1.1 p.171]{POP}. Assume that the action is faithful (so $C$ is almost homogeneous). Let 
\begin{tikzpicture}[baseline=(m.center)]
\matrix(m)[matrix of math nodes,
row sep=2.5em, column sep=2.5em,
text height=1.5ex, text depth=0.25ex,ampersand replacement=\&]
{\widetilde{Z} \& \widetilde{C} \\
Z \& C\\};
\path[->] (m-1-1) edge node[above] {$j$}(m-1-2);
\path[->] (m-1-1) edge node[left] {$\lambda$} (m-2-1);
\path[->] (m-2-1) edge node[below] {$i$} (m-2-2);
\path[->] (m-1-2) edge node[right] {$\nu$} (m-2-2);
\end{tikzpicture}
be the conductor square. Every $\G_{a,k}$-torsor is trivial so, by Theorem \ref{th_general_classification}, we have $\widetilde{C} = \A^1_k$ or $\widetilde{C} = \PP^1_k$ (with the natural action of $\G_{a,k}$). In the first case $\widetilde{C}$ is homogeneous so $C$ is homogeneous too and $C = \A^1_k$. We assume henceforth that $\widetilde{C} = \PP^1_k$. By Lemma \ref{lemma_almost_homogeneous_lift}, $C$ is obtained by pinching a closed subscheme of $\PP^1_k$ supported by the point $\infty$. Thus we write $\widetilde{Z} = \Spec \dfrac{k[u]}{(u^{N})}$ for some $N \geq 0$. Then $A = \mathcal{O}(Z)$ is a subalgebra of $\dfrac{k[u]}{(u^{N})}$ which is $\G_{a,k}$-invariant.

\medskip
An element $a \in \G_{a,k}(k)$ acts on $\PP^1_k$ by the matrix $\begin{pmatrix} 1 & a \\ 0 & 1\end{pmatrix}$, that is, for a point with projective coordinates $[1 : u]$, we have $a \cdot [1 : u] = [1+au : u] = \left[1 : \dfrac{u}{1+au}\right]$. We still denote by $u$ the image of $u$ in $\dfrac{k[u]}{(u^{N})}$. In $\dfrac{k[u]}{(u^{N})}$, the element $\dfrac{u}{1+au}$ is invertible and we have $\dfrac{u}{1+au} = u-au^2 + \cdots + (-1)^{N-2}a^{N-2}u^{N-1}$. Moreover, for $1 \leq i \leq N-1$, we have $a \cdot u^i = (a \cdot u)^i$. Hence in the basis $(1,u,\ldots,u^{N-1})$, the endomorphism of $\dfrac{k[u]}{(u^{N})}$ induced by $a$ has a triangular matrix of the form 
\begin{center}
$\begin{pmatrix} 
1 & & & & & \\
0 & 1 & & & & \\
\vdots & -2a & \ddots & & & \\
\vdots & * & \ddots & \ddots & & \\
\vdots & \vdots & \ddots & \ddots & 1 & \\
0 & * & \cdots &  * & -(N-1)a & 1\\
\end{pmatrix}$
\end{center}

\medskip
The subspaces $k$ and $V = \Vect(u,\ldots,u^{N-1})$ of $\dfrac{k[u]}{(u^{N})}$ are $\G_{a,k}$-stable. We can write $A = k \oplus (A\cap V)$. Let us notice that $A \cap V$  is the unique maximal ideal $\m$ of $A$. Then $A$ is $\G_{a,k}$-stable if and only if $\m$ is stable. The image of $\G_{a,k}$ in $\GL(V)$ is a unipotent group so, by the Lie-Kolchin theorem, every nonzero stable subspace of $V$ contains a stable line. But, since $k$ has characteristic zero, the above matrix shows that the only stable line in $V$ in $\Vect(u^{N-1})$. If $\m = 0$ then $A = k$. We now assume that $\m$ is stable and nonzero. We extend $u^{N-1}$ to a basis $(e_1,\ldots,e_d,u^{N-1})$ of $\m$. By putting the coordinates of the $e_i$ in an echelon form, we can assume that there exist integers $n_1,\ldots,n_{d+1}$ and scalars $b_{i,j} \in k$ such that we have $$1 \leq n_1 < \ldots < n_d < n_{d+1}=N-1$$ and, for $1 \leq i \leq d$, $$e_i = u^{n_i}+b_{i,n_i+1}u^{n_i+1} + \cdots + b_{i,N-2}u^{N-2} + b_{i,N-1}u^{N-1}.$$ Then in the basis $(1,u,\ldots,u^{N-1})$ of $\dfrac{k[u]}{(u^{N})}$, the $n_i-1$ first coordinates of $a \cdot e_i$ equal zero, the $n_i$th one equals $1$ and the $(n_i+1)$th one equals $(-(n_i+1)a + b_{i,n_i+1}$ (which is therefore different from $b_{i,n_i+1}$ if $a \neq 0$). Thus for the images of $e_i$ to be in $\m$, we must have $n_{i+1}=n_i+1$. So $\m = \Vect(u^{n_1},u^{n_1+1},\ldots,u^{N-1})$ and $A= \dfrac{k[u^{n_1},u^{n_1+1},\ldots,u^N]}{(u^N)}$, which conversely is $\G_{a,k}$-stable.

\medskip
The ideal $u^{n_1}$ of $\dfrac{k[u]}{(u^{N})}$ is also an ideal of $\dfrac{k[u^{n_1},u^{n_1+1},\ldots,u^N]}{(u^N)}$. Hence by \cite[Lemme 1.3]{FER} the square of rings 
\begin{center}
\begin{tikzpicture}[baseline=(m.center)]
\matrix(m)[matrix of math nodes,
row sep=3.5em, column sep=3.5em,
text height=3.5ex, text depth=2ex,ampersand replacement=\&]
{\dfrac{k[u^{n_1},u^{n_1+1},\ldots,u^N]}{(u^N)} \& \dfrac{k[u]}{(u^N)} \\
k \simeq \dfrac{k[u^{n_1},u^{n_1+1},\ldots,u^N]/(u^N)}{(u^{n_1})} \& \dfrac{k[u]/(u^N)}{(u^{n_1})} \simeq \dfrac{k[u]}{(u^{n_1})} \\};
\path[right hook ->] (m-1-1) edge (m-1-2);
\path[->>] (m-1-1) edge (m-2-1);
\path[right hook ->] (m-2-1) edge (m-2-2);
\path[->>] (m-1-2) edge (m-2-2);
\end{tikzpicture}
\end{center}
is cartesian. Moreover the morphisms in this square are $\G_{a,k}$-equivariant. So the corresponding square of schemes 
\begin{tikzpicture}[baseline=(m.center)]
\matrix(m)[matrix of math nodes,
row sep=2.5em, column sep=2.5em,
text height=3ex, text depth=2ex,ampersand replacement=\&]
{\Spec \dfrac{k[u]}{(u^{n_1})} \& \widetilde{Z} \\
\Spec k \& Z \\};
\path[->] (m-1-1) edge (m-1-2);
\path[->] (m-1-1) edge node[left] {$\lambda'$} (m-2-1);
\path[->] (m-2-1) edge (m-2-2);
\path[->] (m-1-2) edge node[right] {$\lambda$} (m-2-2);
\end{tikzpicture}
is a pinching diagram where the morphisms are $\G_{a,k}$-equivariant. By concatenation, we get a square 
\begin{tikzpicture}[baseline=(m.center)]
\matrix(m)[matrix of math nodes,
row sep=2.5em, column sep=2.5em,
text height=3ex, text depth=2ex,ampersand replacement=\&]
{\Spec \dfrac{k[u]}{(u^{n_1})} \& \PP^1_k\\
\Spec k \& C \\};
\path[->] (m-1-1) edge node[above] {$j'$} (m-1-2);
\path[->] (m-1-1) edge node[left] {$\lambda'$} (m-2-1);
\path[->] (m-2-1) edge node[below] {$i'$} (m-2-2);
\path[->] (m-1-2) edge node[right] {$\nu$} (m-2-2);
\end{tikzpicture}
which is a pinching diagram where the morphism are $\G_{a,k}$-equivariant. 
\end{proof}

\begin{remark}
The set of the subsemigroups of $(\N,+)$ of the form $\mathfrak{z}_m(\underline{c})$ is countable (since $\mathfrak{z}_m(\underline{c})$ is determined by the tuple $\underline{c} = (c_0,\ldots,c_r)$ and the integer $m$). Hence, contrary to what happened in Remark \ref{rem_continuous_family}, if $k$ has characteristic zero then the family of curves which are almost homogeneous under the action of $\G_{a,k}$ is parameterized by a countable set.
\end{remark}

\begin{corollary}
We assume that $k$ has characteristic zero. Let $C$ be a curve and $\alpha : \G_{a,k} \rtimes \G_{m,k} \times C \to C$ an action. The action is faithful (and $C$ is almost homogeneous) if and only if $C \simeq \A^1_k$ or if there exists $n \geq 0$ such that $C \simeq \PP^1_{k,n}$, and the action is the natural one.
\end{corollary}

\begin{proof}
We use the previous notations. It follows from Theorem \ref{th_general_classification} that if $\G_{a,k} \rtimes \G_{m,k}$ acts faithfully on $C$ then we have $\widetilde{C} \simeq \A^1_k$ or $\PP^1_k$ (with the natural action). The subschemes $\widetilde{Z}$ and $Z$ are stable under the action of $\G_{a,k} \rtimes \G_{m,k}$ so they are stable under the action of the subgroup $\G_{a,k}$. Thus $C$ is one of the curves given in Theorem \ref{th_classification_Popov_Ga}. Conversely, $\G_{a,k} \rtimes \G_{m,k}$ acts faithfully on these curves.
\end{proof}

\subsection{\texorpdfstring{Almost homogeneous curves under the action of a form of $\G_{m,k}$}{Almost homogeneous curves under the action of a form of Gm}}\label{section_arbitrary_Gm}

In this section we first classify the almost homogeneous curves under the action of $\G_{m,k}$. Then we use the link between the forms of $\G_{m,k}$ and the conics to deduce the classification of almost homogeneous curves under the action of a form of $\G_{m,k}$.

\bigskip
We consider again the projective coordinates $[t:u]$ on $\PP^1_k$. If $\mathfrak{z}$ is a subsemigroup of $(\N,+)$ containing $0$ and all the integers that are large enough then there exist a unique integer $m$ and a unique tuple $\underline{c} = (c_0,\ldots,c_p)$ such that $0 = c_0 < \ldots < c_p < m-1$ and $\mathfrak{z} = \{c_0,\ldots,c_p\} \cup \{ r \in \N \mid r \geq m\}$. The subsemigroup $\mathfrak{z}$ is determined by $m$ and $\underline{c}$, and we denote it by $\mathfrak{z}_m(\underline{c})$. 

For any subsemigroups $\mathfrak{z}_m(\underline{c})$ and $\mathfrak{z}_n(\underline{d})$ of $(\N,+)$, we denote respectively by $\A^1_{k,m}(\underline{c})$,  $\PP^1_{k,m,n}(\underline{c},\underline{d})$ and $\PP^1_{k,m,n}(\underline{c},\underline{d})'$ the curves defined by the pinching diagrams 
\begin{center}
\begin{tikzpicture}[baseline=(m.center)]
\matrix(m)[matrix of math nodes,
row sep=3.5em, column sep=2.5em,
text height=1.5ex, text depth=0.25ex,ampersand replacement=\&]
{Z_m \& \A^1_k\\
Z_m(\underline{c}) \& \A^1_{k,m}(\underline{c}) \\};
\path[->] (m-1-1) edge (m-1-2);
\path[->] (m-1-1) edge (m-2-1);
\path[->] (m-2-1) edge (m-2-2);
\path[->] (m-1-2) edge (m-2-2);
\end{tikzpicture}
\begin{tikzpicture}[baseline=(m.center)]  
\matrix(m)[matrix of math nodes,
row sep=3.5em, column sep=2.5em,
text height=1.5ex, text depth=0.25ex,ampersand replacement=\&]
{Z_m \sqcup Z_n \& \PP^1_k\\
Z_m(\underline{c}) \sqcup Z_n(\underline{d}) \& \PP^1_{k,m,n}(\underline{c},\underline{d}) \\};
\path[->] (m-1-1) edge (m-1-2);
\path[->] (m-1-1) edge (m-2-1);
\path[->] (m-2-1) edge (m-2-2);
\path[->] (m-1-2) edge (m-2-2);
\end{tikzpicture}
\begin{tikzpicture}[baseline=(m.center)]
\matrix(m)[matrix of math nodes,
row sep=3.5em, column sep=2.5em,
text height=1.5ex, text depth=0.25ex,ampersand replacement=\&]
{\Spec k \sqcup \Spec k \& \PP^1_{k,m,n}(\underline{c},\underline{d}) \\
\Spec k \& \PP^1_{k,m,n}(\underline{c},\underline{d})' \\};
\path[->] (m-1-1) edge (m-1-2);
\path[->] (m-1-1) edge (m-2-1);
\path[->] (m-2-1) edge (m-2-2);
\path[->] (m-1-2) edge (m-2-2);
\end{tikzpicture}
\end{center}
where $Z_m = \Spec \dfrac{k[t]}{(t^{m})}$ and $Z_n = \Spec \dfrac{k[u]}{(u^{n})}$ are the $m$th and $n$th infinitesimal neighborhoods of the points $0$ and $\infty$, $Z_m(\underline{c}) = \Spec \dfrac{k[t^{c_0},\ldots,t^{c_p},t^m]}{(t^m)}$ and $Z_n(\underline{d}) = \Spec \dfrac{k[u^{d_0},\ldots,u^{d_q},u^n]}{(u^n)}$, and $\Spec k \sqcup \Spec k$ is the reduced subscheme of $Z_m(\underline{c}) \sqcup Z_n(\underline{d})$.

\begin{remark}
For example, $\A^1_k = \A^1_{k,m}(\underline{c})$ and $\PP^1_k = \PP^1_{k,m,n}(\underline{c},\underline{d})$ for $m=n=0$ and $\underline{c} = \underline{d} = \emptyset$.
\end{remark}

\begin{theorem}\label{th_classification_Popov_Gm}
Let $C$ be a curve. The group $\G_{m,k}$ acts non-trivially on $C$ if and only if $C \simeq \A^1 \setminus \{0\}$ or if there exist subsemigroups $\mathfrak{z}_m(\underline{c})$ and $\mathfrak{z}_n(\underline{d})$ of $(\N,+)$ such that $C$ is isomorphic to $\A^1_{k,m}(\underline{c})$, $\PP^1_{k,m,n}(\underline{c},\underline{d})$ or $\PP^1_{k,m,n}(\underline{c},\underline{d})'$.
\end{theorem}

\begin{proof}
The argument is essentially the same as in the proof of Popov's theorem \cite[Th. 1.2 p.171]{POP} for $\A^1_{k,m}(\underline{c})$ and $\PP^1_{k,m,n}(\underline{c},\underline{d})$. For $\PP^1_{k,m,n}(\underline{c},\underline{d})'$ the argument differs from \cite[Th. 1.3 p.175]{POP}; the classification of the curves in terms of conductor squares is obtained rather quickly and then we just need to simplify the diagrams, as in the proof of Theorem \ref{th_classification_Popov_Ga}.

\medskip
It follows from Lemmas \ref{lemma_descent_action} and \ref{lemma_almost_homogeneous_lift} that $\G_{m,k}$ acts on $\A^1_{k,m}(\underline{c})$, $\PP^1_{k,m,n}(\underline{c},\underline{d})$ and $\PP^1_{k,m,n}(\underline{c},\underline{d})'$. Conversely, assume that $\G_{m,k}$ acts non-trivially on $C$. Let 
\begin{tikzpicture}[baseline=(m.center)]
\matrix(m)[matrix of math nodes,
row sep=2.5em, column sep=2.5em,
text height=1.5ex, text depth=0.25ex,ampersand replacement=\&]
{\widetilde{Z} \& \widetilde{C}\\
Z \& C \\};
\path[->] (m-1-1) edge (m-1-2);
\path[->] (m-1-1) edge (m-2-1);
\path[->] (m-2-1) edge (m-2-2);
\path[->] (m-1-2) edge (m-2-2);
\end{tikzpicture}
be the conductor square. By Theorem \ref{th_general_classification} and Lemma \ref{lemma_almost_homogeneous_lift}, $\widetilde{C}$ is isomorphic either to $\A^1_k$ or to $\PP^1_k$.

\medskip
We first treat the case $\widetilde{C} \simeq \A^1_k$. The closed subscheme $\widetilde{Z}$ of $\A^1_k$ is supported by $0$ so we can write $\widetilde{Z} = Z_M = \Spec \dfrac{k[t]}{(t^{M})}$ for some integer $M$. Then $Z = \Spec A$ for some $\G_{m,k}$-invariant subalgebra $A$ of $\dfrac{k[t]}{(t^{M})}$. We still denote by $t$ the image of $t$ in $\dfrac{k[t]}{(t^{M})}$. For $a \in \G_{m,k}(k_s)$ and $0 \leq i \leq M-1$, we have $a \cdot t^i = a^it^i$ so $t^i$ is an eigenvector of weight $i$. Thus the $k_s$-vector space $\dfrac{k_s[t]}{(t^{M})}$ is the direct sum of stable lines for pairwise distinct weights. Hence the subspace $A \otimes_k k_s$ of $\dfrac{k_s[t]}{(t^{M})}$ is spanned by some $t^i$'s. Since it is a subring, if it contains $t^i$ and $t^j$ then it contains $t^{i+j}$ too. Thus there exists a subsemigroup $\mathfrak{z}_m(\underline{c})$ of $(\N,+)$ such that $m \leq M$ and $A \otimes_k k_s$ is spanned by the $t^i$'s for $i \in \mathfrak{z}_m(\underline{c})$, that is, $A \otimes_k k_s = \dfrac{k_s[t^{c_0},\ldots,t^{c_p},t^m,t^{m+1},\ldots]}{(t^M)}$. By Galois descent for vector subspaces, we consequently have $A = \dfrac{k[t^{c_0},\ldots,t^{c_p},t^m,t^{m+1},\ldots]}{(t^M)}$, which conversely is a $\G_{m,k}$-stable subalgebra of $\dfrac{k[t]}{(t^{M})}$. Finally, as in the proof of Theorem \ref{th_classification_Popov_Ga}, we can replace $\dfrac{k[t]}{(t^{M})}$ and $\dfrac{k[t^{c_0},\ldots,t^{c_p},t^m,t^{m+1},\ldots]}{(t^M)}$ with $\dfrac{k[t]}{(t^{m})}$ and $\dfrac{k[t^{c_0},\ldots,t^{c_p},t^m]}{(t^m)}$. Therefore $C \simeq \A^1_{k,m}(\underline{c})$.

\medskip
We now assume $\widetilde{C} \simeq \PP^1_k$. Similarly, the closed subscheme $\widetilde{Z}$ of $\PP^1_k$ is supported by $\{0,\infty\}$ so we can write $\widetilde{Z} = Z_M \sqcup Z_N$ for some integers $M$ and $N$. Then $Z$ is supported by one or two $k$-rational points. 

In the latter case there exist $\G_{m,k}$-invariant subalgebras $A$ and $B$ of $\dfrac{k[t]}{(t^{M})}$ and $\dfrac{k[u]}{(u^{N})}$ such that $Z = \Spec A \sqcup \Spec B$, and by the same arguments as before we get an isomorphism $C \simeq \PP^1_{k,m,n}(\underline{c},\underline{d})$. 

In the former case there exists a $\G_{m,k}$-invariant subalgebra $A$ of $\dfrac{k[t]}{(t^{M})} \times \dfrac{k[u]}{(u^{N})}$ such that $Z = \Spec A$. Since $(t^i,0)$ and $(0,u^j)$ are eigenvectors of respective weights $i$ and $-j$, the subspace $A \otimes_k k_s$ of $\dfrac{k_s[t]}{(t^{M})} \times \dfrac{k_s[u]}{(u^{N})}$ is spanned by some $(t^i,0)$'s and $(0,u^j)$'s with $i \geq 1$ and $j \geq 1$, and a subspace of $k_s \times k_s$ containing $k_s \cdot (1,1)$ (which is $k_s \cdot (1,1)$ or the whole $k_s \times k_s$). As above and by Galois descent for vector subspaces, there exist two subsemigroups $\mathfrak{z}_m(\underline{c})$ and $\mathfrak{z}_n(\underline{d})$ of $(\N,+)$ such that $A$ is equal to 
\begin{center}
$\dfrac{\Vect(t^{c_1},\ldots,t^{c_p},t^m, t^{m+1},\ldots)}{(t^M)} \times \dfrac{\Vect(u^{d_1},\ldots,u^{d_q},u^n,u^{n+1},\ldots)}{(u^N)} \oplus k \cdot (1,1)$

or $\dfrac{k[t^{c_0},\ldots,t^{c_p},t^m, t^{m+1},\ldots]}{(t^M)} \times \dfrac{k[u^{d_0},\ldots,u^{d_q},u^n,u^{n+1},\ldots]}{(u^N)}$.
\end{center}
Moreover $A$ has a unique maximal ideal, so we must have $$A = \dfrac{\Vect(t^{c_1},\ldots,t^{c_p},t^m, t^{m+1},\ldots)}{(t^M)} \times \dfrac{\Vect(u^{d_1},\ldots,u^{d_q},u^n,u^{n+1},\ldots)}{(u^N)} \oplus k \cdot (1,1).$$ By the same argument as above again, we can replace $\dfrac{k[t]}{(t^{M})} \times \dfrac{k[u]}{(u^{N})}$ and $A$ with $\dfrac{k[t]}{(t^{m})} \times \dfrac{k[u]}{(u^{n})}$ and $$A'=\dfrac{\Vect(t^{c_1},\ldots,t^{c_p},t^m)}{(t^m)} \times \dfrac{\Vect(u^{d_1},\ldots,u^{d_q},u^n)}{(u^n)} \oplus k \cdot (1,1).$$ We have a pinching diagram 
\begin{tikzpicture}[baseline=(m.center)]
\matrix(m)[matrix of math nodes,
row sep=2.5em, column sep=2.5em,
text height=1.5ex, text depth=0.25ex,ampersand replacement=\&]
{Z_m \sqcup Z_n \& \PP^1_k\\
\Spec A' \& C \\};
\path[->] (m-1-1) edge (m-1-2);
\path[->] (m-1-1) edge (m-2-1);
\path[->] (m-2-1) edge (m-2-2);
\path[->] (m-1-2) edge (m-2-2);
\end{tikzpicture}.
Let $C'$ be the curve defined by the pinching diagram
\begin{tikzpicture}[baseline=(m.center)]
\matrix(m)[matrix of math nodes,
row sep=2.5em, column sep=2em,
text height=1.5ex, text depth=0.25ex,ampersand replacement=\&]
{Z_m(\underline{c}) \sqcup Z_n(\underline{d}) \& \PP^1_{k,m,n}(\underline{c},\underline{d})\\
\Spec A' \& C' \\};
\path[->] (m-1-1) edge (m-1-2);
\path[->] (m-1-1) edge (m-2-1);
\path[->] (m-2-1) edge (m-2-2);
\path[->] (m-1-2) edge (m-2-2);
\end{tikzpicture}. 
The diagram
\begin{tikzpicture}[baseline=(m.center)]
\matrix(m)[matrix of math nodes,
row sep=2.5em, column sep=2.5em,
text height=1.5ex, text depth=0.25ex,ampersand replacement=\&]
{Z_m \sqcup Z_n \& \PP^1_k \\
Z_m(\underline{c}) \sqcup Z_n(\underline{d}) \& \PP^1_{k,m,n}(\underline{c},\underline{d})\\
\Spec A' \& C' \\};
\path[->] (m-1-1) edge (m-1-2);
\path[->] (m-1-1) edge (m-2-1);
\path[->] (m-2-1) edge (m-2-2);
\path[->] (m-1-2) edge (m-2-2);
\path[->] (m-2-1) edge (m-3-1);
\path[->] (m-3-1) edge (m-3-2);
\path[->] (m-2-2) edge (m-3-2);
\end{tikzpicture}
is commutative and the two squares are cocartesian, so the entire composed square is cocartesian. By unicity of the scheme obtained by pinching, we have $C' \simeq C$. By \cite[Lemme 1.3]{FER}, the square of rings
\begin{center}
\begin{tikzpicture}[baseline=(m.center)]
\matrix(m)[matrix of math nodes,
row sep=3.5em, column sep=3.5em,
text height=3.5ex, text depth=2ex,ampersand replacement=\&]
{A' \& \dfrac{k[t^{c_0},\ldots,t^{c_p},t^m]}{(t^m)} \times \dfrac{k[u^{d_0},\ldots,u^{d_q},u^n]}{(u^n)} \\
k \& k \times k \\};
\path[right hook ->] (m-1-1) edge (m-1-2);
\path[->>] (m-1-1) edge (m-2-1);
\path[right hook ->] (m-2-1) edge (m-2-2);
\path[->>] (m-1-2) edge (m-2-2);
\end{tikzpicture}
\end{center}
is cartesian. Moreover the morphisms in this square are $\G_{m,k}$-equivariant. Therefore we have a pinching diagram
\begin{tikzpicture}[baseline=(m.center)]
\matrix(m)[matrix of math nodes,
row sep=2.5em, column sep=2.5em,
text height=1.5ex, text depth=0.25ex,ampersand replacement=\&]
{\Spec k \sqcup \Spec k \& \PP^1_{k,m,n}(\underline{c},\underline{d})\\
\Spec k \& C \\};
\path[->] (m-1-1) edge (m-1-2);
\path[->] (m-1-1) edge (m-2-1);
\path[->] (m-2-1) edge (m-2-2);
\path[->] (m-1-2) edge (m-2-2);
\end{tikzpicture}.
\end{proof}

Let $G$ be a form of $\G_{m,k}$. Let $\widetilde{C}$ be a smooth projective conic and $\widetilde{P}$ a separable point of degree $2$ of $\widetilde{C}$ such that $G$ is the centralizer of $\widetilde{P}$ in $\Autgp{\widetilde{C}}$ (see Lemma \ref{lemma_conic}). Let $K=\kappa(\widetilde{P})$, $\Gamma = \Gal(K/k) \simeq \Z/2\Z$, $\mathfrak{z}_m(\underline{c})$ a subsemigroup of $(\N,+)$ and $\widetilde{Y}$ the $m$th infinitesimal neighborhood of $\widetilde{P}$ (which is a $G$-stable closed subscheme of $\widetilde{C}$). There exists an isomorphism $\widetilde{C}_K \simeq \PP^1_K$ such that $\widetilde{Y}_K$ is the $m$th infinitesimal neighborhood of $\{0,\infty\}$, that is, $\widetilde{Y}_K = (Z_m)_K \sqcup (Z_m)_K$ (with the previous notations). Set $Y' = (Z_m(\underline{c}))_K \sqcup (Z_m(\underline{c}))_K$. The Galois action of $\Gamma$ on $\PP^1_K$ exchanges the two components of $\widetilde{Y}_K$, as well as the two components of $Y'$, and the natural morphism $\lambda' : (\widetilde{Y})_K \to Y'$ is $\Gamma$-equivariant. Thus, by Galois descent, there exists a unique subalgebra $A$ of $\mathcal{O}(\widetilde{Y})$ such that the scheme $Y = \Spec A$ satisfies $Y_K = Y'$ and that $\lambda'$ is the morphism deduced from the natural morphism $\lambda : \widetilde{Y} \to Y$. Let $\widetilde{C}_m(\widetilde{P},\underline{c})$ and $\widetilde{C}_m(\widetilde{P},\underline{c})'$ be the curves defined by the pinching diagrams 
\begin{center}
\begin{tikzpicture}[baseline=(m.center)]
\matrix(m)[matrix of math nodes,
row sep=2.5em, column sep=2.5em,
text height=1.5ex, text depth=0.25ex,ampersand replacement=\&]
{\widetilde{Y} \& \widetilde{C}\\
Y \& \widetilde{C}_m(\widetilde{P},\underline{c}) \\};
\path[->] (m-1-1) edge (m-1-2);
\path[->] (m-1-1) edge (m-2-1);
\path[->] (m-2-1) edge (m-2-2);
\path[->] (m-1-2) edge (m-2-2);
\end{tikzpicture}
and
\begin{tikzpicture}[baseline=(m.center)]
\matrix(m)[matrix of math nodes,
row sep=2.5em, column sep=2.5em,
text height=1.5ex, text depth=0.25ex,ampersand replacement=\&]
{\Spec K \& \widetilde{C}_m(\widetilde{P},\underline{c}) \\
\Spec k \& \widetilde{C}_m(\widetilde{P},\underline{c})' \\};
\path[->] (m-1-1) edge (m-1-2);
\path[->] (m-1-1) edge (m-2-1);
\path[->] (m-2-1) edge (m-2-2);
\path[->] (m-1-2) edge (m-2-2);
\end{tikzpicture}
\end{center}
Since after field extension we still have pinching diagrams, by unicity we have $\widetilde{C}_m(\widetilde{P},\underline{c})_K \simeq \PP^1_{K,m,m}(\underline{c},\underline{c})$ and $\widetilde{C}_m(\widetilde{P},\underline{c})' \simeq \PP^1_{K,m,m}(\underline{c},\underline{c})'$. We have $G_K \simeq \G_{m,K}$ and the action $\G_{m,K} \times Y' \to Y'$ is $\Gamma$-equivariant so, by Galois descent again, we have an action of $G$ on $Y$ such that $\lambda$ is equivariant. Hence $G$ acts faithfully on $\widetilde{C}_m(\widetilde{P},\underline{c})$ and $\widetilde{C}_m(\widetilde{P},\underline{c})'$.

\begin{theorem}\label{th_classification_Popov_forms_Gm}
Let $C$ be a curve and $\alpha : G \times C \to C$ an action. The group $G$ acts faithfully on $C$ (and $C$ is almost homogeneous) if and only if $C \simeq \widetilde{C} \setminus \{\widetilde{P}\}$ or if there exists a subsemigroup $\mathfrak{z}_m(\underline{c})$ of $(\N,+)$ such that $C$ is isomorphic to $\widetilde{C}_m(\widetilde{P},\underline{c})$ or $\widetilde{C}_m(\widetilde{P},\underline{c})'$, and the action is the natural one.
\end{theorem}

\begin{proof}
Assume that $G$ acts faithfully on $C$. It follows from Theorem \ref{th_general_classification} and Lemmas \ref{lemma_action_regular_completion} and \ref{lemma_almost_homogeneous_lift} that $\widetilde{C}$ is the normalization of $C$ (with the natural action of $G$). Let
\begin{tikzpicture}[baseline=(m.center)]
\matrix(m)[matrix of math nodes,
row sep=2.5em, column sep=2.5em,
text height=1.5ex, text depth=0.25ex,ampersand replacement=\&]
{\widetilde{Z} \& \widetilde{C}\\
Z \& C \\};
\path[->] (m-1-1) edge (m-1-2);
\path[->] (m-1-1) edge (m-2-1);
\path[->] (m-2-1) edge (m-2-2);
\path[->] (m-1-2) edge (m-2-2);
\end{tikzpicture}
be the conductor square. Then $\widetilde{Z}$ is supported by $\widetilde{P}$ and $Z$ is supported by a point $P$ such that $\kappa(P) = k$ or $K$. By Lemma \ref{lemma_conductor_base_change}, the diagram
\begin{tikzpicture}[baseline=(m.center)]
\matrix(m)[matrix of math nodes,
row sep=2.5em, column sep=2.5em,
text height=1.5ex, text depth=0.25ex,ampersand replacement=\&]
{\widetilde{Z}_K \& \PP^1_K\\
Z_K \& C_K \\};
\path[->] (m-1-1) edge (m-1-2);
\path[->] (m-1-1) edge (m-2-1);
\path[->] (m-2-1) edge (m-2-2);
\path[->] (m-1-2) edge (m-2-2);
\end{tikzpicture}
is still the conductor square. 

\medskip
If $\kappa(P) = K$ then $Z_K$ is supported by two $K$-rational points. As in Theorem \ref{th_classification_Popov_Gm}, there exist integers $M$ and $N$ such that $\widetilde{Z}_K = \Spec\left(\dfrac{K[t]}{(t^M)} \times \dfrac{K[u]}{(u^N)}\right)$ and subsemigroups $\mathfrak{z}_m(\underline{c})$ and $\mathfrak{z}_n(\underline{d})$ of $(\N,+)$ such that $Z_K = \Spec\left(\dfrac{K[t^{c_0},\ldots,t^{c_p},t^m, t^{m+1},\ldots]}{(t^M)} \times \dfrac{K[u^{d_0},\ldots,u^{d_q},u^n,u^{n+1},\ldots]}{(u^N)}\right)$. Moreover the group $\Gamma$ exchanges the two components of $\widetilde{Z}_K$, as well as the two components of $Z_K$. Thus we must have $M=N$ and $\mathfrak{z}_m(\underline{c}) = \mathfrak{z}_n(\underline{d})$. As before, the squares in
\begin{center}
\begin{tikzpicture}[baseline=(m.center)]
\matrix(m)[matrix of math nodes,
row sep=2.5em, column sep=2.5em,
text height=1.5ex, text depth=0.25ex,ampersand replacement=\&]
{\widetilde{Y}_K \& \widetilde{Z}_K \& \PP^1_K\\
Y_K \& Z_K \& C_K \\};
\path[->] (m-1-1) edge (m-1-2);
\path[->] (m-1-1) edge (m-2-1);
\path[->] (m-2-1) edge (m-2-2);
\path[->] (m-1-2) edge (m-2-2);
\path[->] (m-1-2) edge (m-1-3);
\path[->] (m-2-2) edge (m-2-3);
\path[->] (m-1-3) edge (m-2-3);
\end{tikzpicture}
\end{center}
are cocartesian, so the entire composed diagram is a pinching diagram. Moreover, the morphisms are $\Gamma$-equivariant. Therefore, by Galois descent, we have $C \simeq \widetilde{C}_m(\widetilde{P},\underline{c})$.

\medskip
If $\kappa(P)=k$ then $Z_K$ is supported by a $K$-rational point. By a similar argument, there exists a subsemigroup $\mathfrak{z}_m(\underline{c})$ of $(\N,+)$ such that $C_K \simeq \PP^1_{K,m,m}(\underline{c},\underline{c})'$ and the morphisms in the pinching diagram defining $\PP^1_{K,m,m}(\underline{c},\underline{c})'$ are $\Gamma$-equivariant. Therefore, we have $C \simeq \widetilde{C}_m(\widetilde{P},\underline{c})'$.
\end{proof}


\section{Equivariant embeddings in a projective space}

\subsection{Linearized line bundles and pinchings}

In this section we describe the linearized line bundles on a scheme obtained by a pinching diagram. We first recall the definition and basic properties of linearized line bundles and the equivariant Picard group. We use \cite{Mumford_GIT} and \cite[Exp. I, \textsection 6]{SGA3-1} as general references.

\bigskip
Let $X$ be a separated scheme of finite type over $k$ and $\pi : L \to X$ a line bundle. Then $\G_{m,k}$ acts on $L$ by multiplication on the fibers.

\begin{definition}
Let $G$ be an algebraic group and $\alpha : G \times X \to X$ an action. We say that $L$ is linearizable if there exists an action $\beta : G \times L \to L$ which commutes with the action on $\G_{m,k}$ and such that $\pi$ is $G$-equivariant. The couple $(L,\beta)$ is called a linearized line bundle. 

A morphism of linearized line bundles is a morphism of line bundles which is $G$-equivariant.
\end{definition}

For example, if $X$ is a smooth curve then its canonical bundle $\omega_X$ is naturally linearized by the action of $G$ on differential forms.

\bigskip

The linearized line bundles $(L,\beta)$ can be characterized in terms of a cocycle condition for an isomorphism between the pullback line bundles $\alpha^*L$ and $\pr_2^*L$ over $G \times X$ (see \cite[pp.30--31]{Mumford_GIT} or \cite[Exp. I, Def. 6.5.1]{SGA3-1}).

\begin{lemma}
The datum of a linearized line bundle $(L,\beta)$ is equivalent to the datum of a line bundle isomorphism $\Phi : \alpha^* L \to \pr_2^* L$ making the following diagram of line bundles over $G \times G \times X$ commute: 
\begin{center}
\begin{tikzpicture}[baseline=(m.center)]
\matrix(m)[matrix of math nodes,
row sep=3.5em, column sep=4.5em,
text height=1.5ex, text depth=0.25ex,ampersand replacement=\&]
{(\id_G \times \alpha))^* \alpha^* L  \& (\mu \times \id_X)^* \alpha^* L  \& (\mu \times \id_X)^* \pr_2^* L \\
(\id_G \times \alpha)^* \pr_2^*L \& \pr_{23}^* \alpha^* L  \& \pr_{23}^* \pr_2^* L \\};
\path[->] (m-1-1) edge node[above] {can} (m-1-2);
\path[->] (m-1-2) edge node[above] {$(\mu \times \id_X)^* \Phi$} (m-1-3);
\path[->] (m-1-1) edge node[left] {$(\id_G \times \alpha)^*\Phi$} (m-2-1);
\path[->] (m-2-1) edge node[below] {can} (m-2-2);
\path[->] (m-2-2) edge node[below] {$\pr_{23}^*\Phi$} (m-2-3);
\path[->] (m-1-3) edge node[right] {can} (m-2-3);
\end{tikzpicture}
\end{center} 
\end{lemma}

\begin{remark}\label{rem_linearizable_reduced}
\begin{enumerate}[wide, labelwidth=!, labelindent=0pt, label=\roman*)]
\item In this setting, a morphism of linearized line bundles $(L_1,\Phi_{L_1}) \to (L_2,\Phi_{L_2})$ is a morphism of line bundles $\phi : L_1 \to L_2$ making the diagram
\begin{tikzpicture}[baseline=(m.center)]
\matrix(m)[matrix of math nodes,
row sep=2.5em, column sep=2.5em,
text height=1.5ex, text depth=0.25ex,ampersand replacement=\&]
{\alpha^*L_1 \& \pr_2^*L_1\\
\alpha^*L_2 \& \pr_2^*L_2\\};
\path[->] (m-1-1) edge node[above] {$\Phi_{L_1}$} (m-1-2);
\path[->] (m-1-1) edge node[left] {$\alpha^*\phi$} (m-2-1);
\path[->] (m-2-1) edge node[below] {$\Phi_{L_2}$} (m-2-2);
\path[->] (m-1-2) edge node[right] {$\pr_2^*\phi$} (m-2-2);
\end{tikzpicture}
commute.
\item By \cite[lemma 2.9]{BRI_lin}, if $X$ is reduced then $L$ is linearizable if and only if $\alpha^*L$ and $\pr_2^*L$ are isomorphic (the isomorphism need not satisfy the cocycle condition).
\item More generally, an arbitrary $\mathcal{O}_X$-module $\mathcal{L}$ is said to be linearizable if there exists an isomorphism $\Phi : \alpha^*\mathcal{L} \to \pr_2^*\mathcal{L}$ satisfying the above cocycle condition (see \cite[Exp. I, Def. 6.5.1]{SGA3-1}).
\end{enumerate}
\end{remark}

If $L$ and $L'$ are linearizable line bundles over $X$ then we have isomorphisms $\Phi : \alpha^* L \to \pr_2^* L$ and $\Phi' : \alpha^* L' \to \pr_2^* L$ satisfying the cocycle condition. Then $\Phi \otimes \Phi' : \alpha^* (L \otimes L') = \alpha^* L \otimes \alpha^* L' \to \pr_2^* L \otimes \pr_2^* L' = \pr_2^*(L \otimes L')$ is an isomorphism satisfying the cocycle condition too, so $L \otimes L'$ is linearizable. Similarly $L^{-1}$ is linearizable. 

\begin{definition}
We denote by $\Pic^G(X)$ the abelian group of isomorphism classes of linearized line bundles over $X$. It is called the equivariant Picard group.
\end{definition}

The equivariant Picard group of a $G$-torsor over $k$ is trivial. More generally, we have the following result.

\begin{lemma}
\cite[p.32]{Mumford_GIT}\label{lemma_PicG_torsor} Let $\pi : X \to Y$ be a $G$-torsor. The pullback by $f$ induces an isomorphism $\Pic Y \simeq \Pic^G(X)$.
\end{lemma}

The linearized line bundles give equivariant embeddings in a projective space.

\begin{lemma}
\cite[Prop. 1.7 p.35]{Mumford_GIT}Let \label{lemma_linearization_embedding}$G$ be an algebraic group, $X$ a quasi-projective scheme over $k$, $L$ a very ample line bundle over $X$ and $\alpha : G \times X \to X$ an action. The following assertions are equivalent:
\begin{enumerate}[label=\roman*)]
 \item The line bundle $L$ is linearizable.
 \item There exists an equivariant immersion of $X$ in the projectivization of a finite-dimensional $G$-module.
 \item There exists an action of $G$ on a projective space $\PP^n_k$ and an equivariant immersion $i : X \to \PP^n_k$ such that $\mathcal{O}_{\PP^n}(1)$ is linearizable and $L = i^* \mathcal{O}_{\PP^n}(1)$.
\end{enumerate}
\end{lemma}

The following result on normal schemes is well-known.

\begin{proposition}
\cite[Th. 2.14]{BRI_lin} Let $G$ be a smooth connected linear algebraic group, $X$ a variety and $\alpha : G \times X \to X$ an action. If $X$ is normal then there exists an integer $n \geq 1$ such that for every line bundle $L$ over $X$, $L^{\otimes n}$ is linearizable.
\end{proposition}

We want an analogous result without the hypothesis of normality in case $G$ is a form of $\G_{a,k}$ in prime characteristic. We first need to understand the obstruction for a line bundle (over a separated reduced scheme of finite type $X$) to be linearizable.

\begin{proposition}
\cite[Prop. 2.10]{BRI_lin} Let \label{prop_PicG_exact_seq}$G$ be a smooth linear algebraic group, $X$ a separated reduced scheme of finite type over $k$ and $\alpha : G \times X \to X$ an action. There is an exact sequence of abstract groups $$1 \to \mathcal{O}(X)^{\times^G} \to \mathcal{O}(X)^\times \to \widehat{G}(X) \to \Pic^G(X) \to \Pic(X) \xrightarrow{\alpha^*} \Pic(G \times X)/\pr_2^*\Pic(X)$$ is exact, where $\Pic^G(X) \to \Pic(X)$ is the forgetful morphism which associates with a linearized line bundle $(L,\Phi)$ the line bundle $L$, and $\mathcal{O}(X)^{\times^G}$ denotes the subgroup of $\mathcal{O}(X)^\times$ consisting of $G$-invariant elements. In particular, if the group $X(G)$ of characters of $G_{k_s}$ is trivial then the forgetful morphism $\Pic^G(X) \to \Pic(X)$ is injective (that is, a line bundle has at most one linearization).
\end{proposition}

\begin{proposition}\label{prop_form_Ga_linearizable_power}
We assume that $k$ has characteristic $p > 0$. Let $G$ be a form of $\G_{a,k}$, $X$ a variety and $\alpha : G \times X \to X$ an action. There exists an integer $r \geq 0$ such that for every line bundle $L$ over $X$, $L^{\otimes p^r}$ is linearizable.
\end{proposition}

\begin{proof}
Because of the exact sequence of Proposition \ref{prop_PicG_exact_seq}, it suffices to show that the group $\Pic(G \times X)/\pr_2^*\Pic(X)$ is $p^r$-torsion.

Let us show that we can restrict to the case $G = \G_{a,k}$. Let $K/k$ be a finite purely inseparable extension such that $G_K \simeq \G_{a,K}$ (see Proposition \ref{prop_Russell}). The degree $d$ of this extension is a power of $p$. Let us assume that there exists an integer $s \geq 0$ such that the group $\Pic(G_K \times X_K)/\pr_2^*\Pic(X_K)$ is $p^s$-torsion. Let $L$ be a line bundle over $G \times X$. By hypothesis, there exists a line bundle $M$ over $X_K$ such that $(L_K)^{\otimes p^s} \simeq \pr_2^*M$. Taking norms, we get $L^{\otimes dp^s} \simeq N_{K/k}\left((L_K)^{\otimes p^s}\right) \simeq N_{K/k}\left(\pr_2^*M\right) \simeq \pr_2^*(N_{K/k}(M))$ (see \cite[II, Prop. 6.5.8]{EGA}). Thus the group $\Pic(G \times X)/\pr_2^*\Pic(X)$ is $dp^s$-torsion. 

\medskip
We now assume $G = \G_{a,k}$. Let $\sigma : X^+ \to X$ be the seminormalization. By Lemma \ref{lemma_Pic_homotopy}, the pullback morphism $\pr_2^* : \Pic(X^+) \to \Pic(G\times X^+)$ is an isomorphism. Thus the following diagram is commutative and exact in rows: 
\begin{center}
\begin{tikzpicture}[baseline=(m.center)]
\matrix(m)[matrix of math nodes,
row sep=2.5em, column sep=2.5em,
text height=1.5ex, text depth=0.25ex,ampersand replacement=\&]
{0 \& \Pic(X) \& \Pic(G \times X) \& \Pic(G \times X)/\pr_2^*\Pic(X) \& 0 \\
0 \& \Pic(X^+) \& \Pic(G \times X^+) \& 0 \& \\};
\path[->] (m-1-1) edge (m-1-2);
\path[->] (m-1-2) edge node[above] {$\pr_2^*$} (m-1-3);
\path[->] (m-1-3) edge (m-1-4);
\path[->] (m-1-4) edge (m-1-5);
\path[->] (m-2-1) edge (m-2-2);
\path[->] (m-2-2) edge node[below] {$\pr_2^*$} (m-2-3);
\path[->] (m-2-3) edge (m-2-4);
\path[->] (m-1-2) edge node[left] {$\sigma^*$} (m-2-2);
\path[->] (m-1-3) edge node[left] {$(\id_G \times \sigma)^*$} (m-2-3);
\path[->] (m-1-4) edge (m-2-4);
\end{tikzpicture}
\end{center}
Hence the snake lemma gives an exact sequence $$\ker(\id_G \times \sigma)^* \to \Pic(G \times X)/\pr_2^*\Pic(X) \to \coker \sigma^*.$$ 
Moreover the morphism $\alpha : G \times X \to X$ is smooth (because $G$ is smooth) so, by Lemma \ref{lemma_SN_base_change}, the scheme $G \times X^+$ is seminormal and $\id_G \times \sigma : G \times X^+ \to G \times X$ is the seminormalization. Thus it follows from \cite[Lemma 4.11]{BRI_lin} that there exists an integer $n \geq 0$ such that $\ker(\id_G \times \sigma)^*$ and $\coker \sigma^*$ are $p^n$-torsion. Hence $\Pic(G \times X)/\pr_2^*\Pic(X)$ is $p^{2n}$-torsion.
\end{proof}

\begin{remark}
\begin{enumerate}[wide, labelwidth=!, labelindent=0pt, label=\roman*)] 
 \item This result is not true if $k$ has characteristic zero. Indeed, $\G_{a,k}$ acts on the cuspidal curve $C$ obtained by pinching the point $\infty$ of $\PP^1_k$ onto $\Spec k[\varepsilon]/(\varepsilon^2)$ (which can be explicitly realized for example as the curve with homogeneous equation $y^3=x^2z$ in $\PP^2_k$). By \cite[Ex. 2.16]{BRI_lin}, a line bundle over $C$ is linearizable if and only if it has degree $0$.
 \item As a particular case of Proposition \ref{prop_form_Ga_linearizable_power}, if $G$ is a non-trivial form of $\G_{a,k}$ then its regular completion $C$ can be embedded equivariantly in the projectivization of a $G$-module, but $C$ is not smooth. On the opposite, if $k$ has characteristic zero then the closure of any orbit in the projectivization of a $\G_{a,k}$-module is smooth (see \cite[Lemma 2.4]{BRIFU}).
\end{enumerate} 
\end{remark}

If
\begin{tikzpicture}[baseline=(m.center)]
\matrix(m)[matrix of math nodes,
row sep=2.5em, column sep=2.5em,
text height=1.5ex, text depth=0.25ex,ampersand replacement=\&]
{\widetilde{Z} \& \widetilde{X} \\
Z \& X\\};
\path[->] (m-1-1) edge node[above] {$j$}(m-1-2);
\path[->] (m-1-1) edge node[left] {$\lambda$} (m-2-1);
\path[->] (m-2-1) edge node[below] {$i$} (m-2-2);
\path[->] (m-1-2) edge node[right] {$\nu$} (m-2-2);
\end{tikzpicture}
is a pinching diagram then Charles Weibel described the line bundles over $X$ in terms of line bundles over $\widetilde{X}$ and $Z$.

\begin{proposition}\cite[Prop. 7.8]{WEIB}
We have an exact sequence of abstract groups 
$$1 \to \mathcal{O}(X)^\times \to \mathcal{O}(\widetilde{X})^\times \times \mathcal{O}(Z)^\times \to \mathcal{O}(\widetilde{Z})^\times \to \Pic X \to \Pic \widetilde{X} \times \Pic Z \to \Pic \widetilde{Z}$$
induced by pullbacks and a connecting morphism $\mathcal{O}(\widetilde{Z})^\times \to \Pic X$. This sequence called the Units-Pic sequence.
\end{proposition}

Loosely speaking, this means that, up to isomorphism, line bundles over $X$ correspond to line bundles over $\widetilde{X}$ equipped with trivializations along fibers of $\lambda$.

This can be reformulated in categorical terms. For any scheme $Y$, we denote by $\mathcal{D}_Y$ the category of line bundles over $Y$. With the functors $j^* : \mathcal{D}_{\widetilde{X}} \to \mathcal{D}_{\widetilde{Z}}$ et $\lambda^* : \mathcal{D}_{Z} \to \mathcal{D}_{\widetilde{Z}}$ we can form the fiber product category $\mathcal{D}_{\widetilde{X}} \times_{\mathcal{D}_{\widetilde{Z}}} \mathcal{D}_{Z}$. Its objects are the triples $(\widetilde{L},M,\sigma)$ where $\widetilde{L}$ and $M$ are line bundles over $\widetilde{X}$ and $Z$, and $\sigma$ is an isomorphism $j^*\widetilde{L} \simeq \lambda^* M$ of line bundles over $Z$. A morphism between two triples $(\widetilde{L}_1,M_1,\sigma_1)$ and $(\widetilde{L}_2,M_2,\sigma_2)$ is given by two morphisms $\widetilde{\phi} : \widetilde{L}_1 \to \widetilde{L}_2$ and $\psi : M_1 \to M_2$ such that $\sigma_2 \circ (j^*\widetilde{\phi}) = (\lambda^*\psi) \circ \sigma_1$.

\begin{proposition}
\cite[Th. 3.13]{HOW}\label{prop_pinching_line_bundles} The functor $$T : \fonction{\mathcal{D}_X}{\mathcal{D}_{\widetilde{X}} \times_{\mathcal{D}_{\widetilde{Z}}} \mathcal{D}_{Z}}{L}{(\nu^*L,i^*L,can)}$$ is a equivalence of categories (where $can$ is the canonical isomorphism $j^*\nu^*L \simeq \lambda^*i^*L$).
\end{proposition}

Let $S$ be the quasi-inverse of $T$. For $(\widetilde{L},M,\sigma)$ in $\mathcal{D}_{\widetilde{X}} \times_{\mathcal{D}_{\widetilde{Z}}} \mathcal{D}_{Z}$, the line bundle $L = S(\widetilde{L},M,\sigma)$ is determined (up to isomorphism) by the existence of isomorphisms $\phi_\nu : \nu^*L \to \widetilde{L}$ and $\phi_i : i^*L \to M$ making the diagram
\begin{tikzpicture}[baseline=(m.center)]
\matrix(m)[matrix of math nodes,
row sep=2.5em, column sep=2.5em,
text height=1.5ex, text depth=0.25ex,ampersand replacement=\&]
{j^*\nu^*L \& j^*\widetilde{L} \\
\lambda^*i^*L \& \lambda^*M\\};
\path[->] (m-1-1) edge node[above] {$j^*\phi_\nu$}(m-1-2);
\path[->] (m-1-1) edge node[left] {$\textrm{can}$} (m-2-1);
\path[->] (m-2-1) edge node[below] {$\lambda^*\phi_i$} (m-2-2);
\path[->] (m-1-2) edge node[right] {$\sigma$} (m-2-2);
\end{tikzpicture}
commute.

\bigskip
We can do the same for linearized line bundles. Let $G$ be a smooth connected algebraic group. For any scheme $Y$ which is separated and of finite type over $k$, and any action $\alpha : G \times Y \to Y$, we denote by $\mathcal{D}_Y^G$ the category of linearized line bundles over $Y$. Let $Y'$ be another such scheme with an action $\alpha' : G \times Y' \to Y'$ and $f : Y' \to Y$ an equivariant morphism. For any $(L,\Phi_L)$ in $\mathcal{D}_Y^G$, the line bundle $f^*L$ over $Y'$ is equipped with the linearization $\Phi_{f^*L}$ given by the diagram 
\begin{tikzpicture}[baseline=(m.center)]
\matrix(m)[matrix of math nodes,
row sep=2.5em, column sep=4.5em,
text height=1.5ex, text depth=0.25ex,ampersand replacement=\&]
{\alpha'^*f^*L \& \pr_2^*f^*L \\
(\id_G \times f)^*\alpha^*L \& (\id_G \times f)^*\pr_2^*L\\};
\path[->] (m-1-1) edge node[above] {$\Phi_{f^*L}$}(m-1-2);
\path[->] (m-1-1) edge node[left] {$\textrm{can}$} (m-2-1);
\path[->] (m-2-1) edge node[below] {$(\id_G \times f)^*\Phi_L$} (m-2-2);
\path[->] (m-1-2) edge node[right] {$\textrm{can}$} (m-2-2);
\end{tikzpicture}.
If $\phi : L_1 \to L_2$ is a morphism of linearized line bundles over $Y$ then $f^*\phi : f^*L_1 \to f^*L_2$ is a morphism of linearized line bundles over $Y'$. This yields a functor $f^* : \mathcal{D}_Y^G \to \mathcal{D}_{Y'}^G$.

We assume that $X$ is separated and of finite type over $k$ (thus so are $\widetilde{X}$, $Z$ and $\widetilde{Z}$). Let $\alpha : G \times X \to X$ and $\widetilde{\alpha} : G \times \widetilde{X} \to \widetilde{X}$ be two actions such that $Z$ and $\widetilde{Z}$ are $G$-stable and $\nu$ is $G$-equivariant (so $\lambda$, $i$ and $j$ are equivariant too). Let $\beta : G \times Z \to Z$ and $\widetilde{\beta} : G \times \widetilde{Z} \to \widetilde{Z}$ be the induced actions. With the functors $j^* : \mathcal{D}_{\widetilde{X}}^G \to \mathcal{D}_{\widetilde{Z}}^G$ et $\lambda^* : \mathcal{D}_{Z}^G \to \mathcal{D}_{\widetilde{Z}}^G$ we can form the fiber product category $\mathcal{D}_{\widetilde{X}}^G \times_{\mathcal{D}_{\widetilde{Z}}^G} \mathcal{D}_{Z}^G$. Its objects are the tuples $(\widetilde{L},\Phi_{\widetilde{L}},M,\Phi_M,\sigma)$ where $(\widetilde{L},\Phi_{\widetilde{L}})$ and $(M,\Phi_M)$ are linearized line bundles over $\widetilde{X}$ and $Z$, and $\sigma$ is an isomorphism $j^*\widetilde{L} \simeq \lambda^* M$ of linearized line bundles over $Z$.

\begin{proposition}
The \label{prop_pinching_linearized_line_bundles}functor $$T^G : \fonction{\mathcal{D}_X^G}{\mathcal{D}_{\widetilde{X}}^G \times_{\mathcal{D}_{\widetilde{Z}}^G} \mathcal{D}_{Z}^G}{(L,\Phi_L)}{(\nu^*L,\Phi_{\nu^*L},i^*L,\Phi_{i^*L},can)}$$ is a equivalence of categories (where $can$ is the canonical isomorphism $j^*\nu^*L \simeq \lambda^*i^*L$).
\end{proposition}

\begin{proof}
It is immediate that the functor $T^G$ is well-defined on objects (and its definition on morphisms is the obvious one).

By proposition \ref{prop_pinching_line_bundles}, the functor $T^G$ is faithful. Let us show that it is full. Let $(L_1,\Phi_{L_1})$ and $(L_2,\Phi_{L_2})$ be two linearized line bundles over $X$. A morphism $T^G(L_1,\Phi_{L_1}) \to T^G(L_2,\Phi_{L_2})$ is given by two morphisms of linearized line bundles $\widetilde{\phi} : (\nu^*L_1,\Phi_{\nu^*L_1}) \to (\nu^*L_2,\Phi_{\nu^*L_2})$ and $\psi : (i^*L_1,\Phi_{i^*L_1}) \to (i^*L_2,\Phi_{i^*L_2})$ such that $\sigma_2 \circ (j^*\widetilde{\phi}) = (\lambda^*\psi) \circ \sigma_1$ (where $\sigma_1$ and $\sigma_2$ are the canonical isomorphisms $j^*\nu^*L_1 \simeq \lambda^*i^*L_1$ and $j^*\nu^*L_2 \simeq \lambda^*i^*L_2$). By proposition \ref{prop_pinching_line_bundles}, there exists a morphism of line bundles $\phi : L_1 \to L_2$ such that $\widetilde{\phi} = \nu^*\phi$ and $\psi = i^*\phi$. It remains to show that $\phi$ is a morphism of linearized line bundles. From the definition of the pullback linearizations $\Phi_{\nu^*L_1}$ and $\Phi_{\nu^*L_2}$, the equality $\widetilde{\phi} = \nu^*\phi$, the fact that $\widetilde{\phi}$ is a morphism of linearized line bundles and the different canonical isomorphisms, it follows readily that the diagram 
\begin{center}
\begin{tikzpicture}[baseline=(m.center)]
\matrix(m)[matrix of math nodes,
row sep=2.5em, column sep=5em,
text height=1.5ex, text depth=0.25ex,ampersand replacement=\&]
{(\id_G \times \nu)^*\alpha^*L_1 \& (\id_G \times \nu)^*\pr_2^*L_1\\
(\id_G \times \nu)^*\alpha^*L_2 \& (\id_G \times \nu)^*\pr_2^*L_2\\};
\path[->] (m-1-1) edge node[above] {$(\id_G \times \nu)^*\Phi_{L_1}$} (m-1-2);
\path[->] (m-1-1) edge node[left] {$(\id_G \times \nu)^*\alpha^*\phi$} (m-2-1);
\path[->] (m-2-1) edge node[below] {$(\id_G \times \nu)^*\Phi_{L_2}$} (m-2-2);
\path[->] (m-1-2) edge node[right] {$(\id_G \times \nu)^*\pr_2^*\phi$} (m-2-2);
\end{tikzpicture}
\end{center}
of line bundles over $G \times \widetilde{X}$ commutes (and similarly for the pullback by $\id_G \times i$). Moreover the square 
\begin{tikzpicture}[baseline=(m.center)]
\matrix(m)[matrix of math nodes,
row sep=2.5em, column sep=2.5em,
text height=1.5ex, text depth=0.25ex,ampersand replacement=\&]
{G\times \widetilde{Z} \& G\times \widetilde{Z} \\
G\times Z \& G\times X\\};
\path[->] (m-1-1) edge node[above] {$\id\times j$}(m-1-2);
\path[->] (m-1-1) edge node[left] {$\id\times \lambda$} (m-2-1);
\path[->] (m-2-1) edge node[below] {$\id\times i$} (m-2-2);
\path[->] (m-1-2) edge node[right] {$\id\times \nu$} (m-2-2);
\end{tikzpicture}
is a pinching diagram (see \cite[Th. 3.11]{HOW}) so we can apply Proposition \ref{prop_pinching_line_bundles} to it. Thus the corresponding functor $$T' : \mathcal{D}_{G\times X} \to \mathcal{D}_{G \times \widetilde{X}} \times_{\mathcal{D}_{G \times \widetilde{Z}}} \mathcal{D}_{G \times Z}$$ is (fully) faithful. Consequently the diagram
\begin{tikzpicture}[baseline=(m.center)]
\matrix(m)[matrix of math nodes,
row sep=2.5em, column sep=2.5em,
text height=1.5ex, text depth=0.25ex,ampersand replacement=\&]
{\alpha^*L_1 \& \pr_2^*L_1\\
\alpha^*L_2 \& \pr_2^*L_2\\};
\path[->] (m-1-1) edge node[above] {$\Phi_{L_1}$} (m-1-2);
\path[->] (m-1-1) edge node[left] {$\alpha^*\phi$} (m-2-1);
\path[->] (m-2-1) edge node[below] {$\Phi_{L_2}$} (m-2-2);
\path[->] (m-1-2) edge node[right] {$\pr_2^*\phi$} (m-2-2);
\end{tikzpicture}
commutes, that is, $\phi$ is a morphism of linearized line bundles.

\medskip
Finally, let us show that the functor $T^G$ is essentially surjective. Let $(\widetilde{L},\Phi_{\widetilde{L}},M,\Phi_M,\sigma)$ be an object of $\mathcal{D}_{\widetilde{X}}^G \times_{\mathcal{D}_{\widetilde{Z}}^G} \mathcal{D}_{Z}^G$. By Proposition \ref{prop_pinching_line_bundles} again, there exists a line bundle $L$ over $X$ and isomorphisms $\phi_\nu : \nu^*L \to \widetilde{L}$ and $\phi_i : i^*L \to M$ making the diagram
\begin{tikzpicture}[baseline=(m.center)]
\matrix(m)[matrix of math nodes,
row sep=2.5em, column sep=2.5em,
text height=1.5ex, text depth=0.25ex,ampersand replacement=\&]
{j^*\nu^*L \& j^*\widetilde{L} \\
\lambda^*i^*L \& \lambda^*M\\};
\path[->] (m-1-1) edge node[above] {$j^*\phi_\nu$}(m-1-2);
\path[->] (m-1-1) edge node[left] {$\textrm{can}$} (m-2-1);
\path[->] (m-2-1) edge node[below] {$\lambda^*\phi_i$} (m-2-2);
\path[->] (m-1-2) edge node[right] {$\sigma$} (m-2-2);
\end{tikzpicture}
commute. We have isomorphisms $$\sigma_1 : (\id\times j)^*\widetilde{\alpha}^*\widetilde{L} \xrightarrow{\textrm{can}} \widetilde{\beta}^*j^* \widetilde{L} \xrightarrow{\widetilde{\beta}^*\sigma} \widetilde{\beta}^*\lambda^*M \xrightarrow{\textrm{can}} (\id \times \lambda)^*\beta^* M$$  $$\sigma_2 : (\id\times j)^*\pr_2^*\widetilde{L} \xrightarrow{\textrm{can}} \pr_2^*j^*\widetilde{L} \xrightarrow{\pr_2^*\sigma} \pr_2^*\lambda^*M \xrightarrow{\textrm{can}} (\id \times \lambda)^*\pr_2^* M$$ of line bundles over $G \times \widetilde{Z}$. Since $\sigma : j^* \widetilde{L} \to \lambda^*M$ is an isomorphism of linearized line bundles, $\Phi_{\widetilde{L}}$ and $\Phi_M$ induce an isomorphism $(\widetilde{\alpha}^* \widetilde{L}, \beta^*M, \sigma_1) \to (\pr_2^*\widetilde{L},\pr_2^*M,\sigma_2)$ in the category $\mathcal{D}_{G \times \widetilde{X}} \times_{\mathcal{D}_{G \times \widetilde{Z}}} \mathcal{D}_{G \times Z}$. As the functor $T'$ defined above is fully faithful, there exists an isomorphism $\Phi_L : \alpha^*L \to \pr_2^*L$ such that $\Phi_{\widetilde{L}}$ is the composite $$\widetilde{\alpha}^*\widetilde{L} \xrightarrow{\widetilde{\alpha}^*\phi_\nu^{-1}} \widetilde{\alpha}^*\nu^*L \xrightarrow{\textrm{can}} (\id\times \nu)^* \alpha^*L \xrightarrow{(\id\times \nu)^*\Phi_L} (\id\times \nu)^*\pr_2^* L \xrightarrow{\textrm{can}} \pr_2^*\nu^*L \xrightarrow{\pr_2^*\phi_\nu} \pr_2^*\widetilde{L}$$ (and similarly for $\Phi_M$). It remains to show that $\Phi_L$ is a linearization, because then $\phi_\nu$ and $\phi_i$ give an isomorphism $T^G(L,\Phi_L) \to (\widetilde{L},\Phi_{\widetilde{L}},M,\Phi_M,\sigma)$. We have to show that $\Phi_L$ satisfies the cocycle condition, that is to say that the diagram
\begin{center}
\begin{tikzpicture}[baseline=(m.center)]
\matrix(m)[matrix of math nodes,
row sep=3.5em, column sep=4.5em,
text height=1.5ex, text depth=0.25ex,ampersand replacement=\&]
{(\id_G \times \alpha))^* \alpha^* L  \& (\mu \times \id_X)^* \alpha^* L  \& (\mu \times \id_X)^* \pr_2^* L \\
(\id_G \times \alpha)^* \pr_2^*L \& \pr_{23}^* \alpha^* L  \& \pr_{23}^* \pr_2^* L \\};
\path[->] (m-1-1) edge node[above] {can} (m-1-2);
\path[->] (m-1-2) edge node[above] {$(\mu \times \id_X)^* \Phi_L$} (m-1-3);
\path[->] (m-1-1) edge node[left] {$(\id_G \times \alpha)^*\Phi_L$} (m-2-1);
\path[->] (m-2-1) edge node[below] {can} (m-2-2);
\path[->] (m-2-2) edge node[below] {$\pr_{23}^*\Phi_L$} (m-2-3);
\path[->] (m-1-3) edge node[right] {can} (m-2-3);
\end{tikzpicture}
\end{center}
commutes. It is easily checked that the two diagrams obtained by pulling back by $\id_G \times \id_G \times \nu$ and by $\id_G \times \id_G \times i$ are commutative because $\Phi_{\widetilde{L}}$ and $\Phi_M$ are linearizations. But, by the same argument as before, the functor $$\mathcal{D}_{G\times G\times X} \to \mathcal{D}_{G\times G \times \widetilde{X}} \times_{\mathcal{D}_{G\times G \times \widetilde{Z}}} \mathcal{D}_{G\times G \times Z}$$ is faithful. Therefore the diagram for $\Phi_L$ is indeed commutative.
\end{proof}

In turn, this equivalence of categories yields an analogue of the Units-Pic sequence.

\begin{corollary}
We have an exact sequence of abstract groups
$$1 \to \mathcal{O}(X)^{\times^G} \to \mathcal{O}(\widetilde{X})^{\times^G} \times \mathcal{O}(Z)^{\times^G} \to \mathcal{O}(\widetilde{Z})^{\times^G} \to \Pic^G(X) \to \Pic^G(\widetilde{X}) \times \Pic^G(Z) \to \Pic^G(\widetilde{Z})$$
induced by pullbacks and a connecting morphism $\delta : \mathcal{O}(\widetilde{Z})^{\times^G} \to \Pic^G(X)$. We call it ``the equivariant Units-Pic sequence''.
\end{corollary}

\begin{proof}
We identify $\mathcal{O}(\widetilde{Z})^\times$ with the group of automorphisms of the trivial line bundle $\widetilde{Z} \times \A^1_k$. Then the automorphisms of $(\widetilde{Z} \times \A^1_k, triv)$ (where $triv$ is the trivial linearization) correspond to the elements of $\mathcal{O}(\widetilde{Z})^{\times^G}$. The morphism $\delta$ maps $\sigma \in \mathcal{O}(\widetilde{Z})^{\times^G}$ to the isomorphism class of linearized line bundles over $X$ corresponding (in view of Proposition \ref{prop_pinching_linearized_line_bundles}) to the isomorphism class of the tuple $(\widetilde{X} \times \A^1_k, triv, Z \times \A^1_k, triv, \sigma)$. By construction, $\delta(\sigma)$ is in the kernel of the morphism $\Pic^G(X) \to \Pic^G(\widetilde{X}) \times \Pic^G(Z)$. Conversely, an element $(L,\Phi_L) \in \ker \delta$ corresponds to a tuple $(\widetilde{X} \times \A^1_k, triv, Z \times \A^1_k, triv, \sigma)$ and $\sigma$ is an automorphism of $(\widetilde{Z} \times \A^1_k, triv)$, so $\sigma \in \mathcal{O}(\widetilde{Z})^{\times^G}$ and $(L,\Phi_L) = \delta(\sigma)$.

\medskip
Let $\sigma \in \mathcal{O}(\widetilde{Z})^\times$, viewed as an automorphism of $(\widetilde{Z} \times \A^1_k,triv)$. Then $\sigma$ is in $\ker \delta$ if an only if there exists an isomorphism $(\widetilde{X} \times \A^1_k, triv, Z \times \A^1_k, triv, \sigma) \to (\widetilde{X} \times \A^1_k, triv, Z \times \A^1_k, triv, can)$. Such an isomorphism is given by two automorphisms $\widetilde{\phi}$ of $(\widetilde{X} \times \A^1_k,triv)$ and $\psi$ of $(Z \times \A^1_k,triv)$ such that $can \circ (j^*\widetilde{\phi}) = (\lambda^*\psi)\circ \sigma$. Viewing $\widetilde{\phi}$ and $\psi$ as elements of $\mathcal{O}(\widetilde{X})^{\times^G}$ and $\mathcal{O}(Z)^{\times^G}$, this condition precisely means that $\sigma$ is the image of $(\widetilde{\phi},\psi)$ by the morphism $\mathcal{O}(\widetilde{X})^{\times^G} \times \mathcal{O}(Z)^{\times^G} \to \mathcal{O}(\widetilde{Z})^{\times^G}$. So the sequence is exact at $\mathcal{O}(\widetilde{Z})^\times$.

\medskip
The image of $\mathcal{O}(X)^{\times^G} \to \mathcal{O}(\widetilde{X})^{\times^G} \times \mathcal{O}(Z)^{\times^G}$ is contained in the kernel of $\mathcal{O}(\widetilde{X})^{\times^G} \times \mathcal{O}(Z)^{\times^G} \to \mathcal{O}(\widetilde{Z})^{\times^G}$. Conversely, if $(\widetilde{\phi},\psi)$ is in the kernel then, by the Units-Pic sequence, it is the image of an element $\phi \in \mathcal{O}(X)^\times$. The morphism $\id_G \times \nu : G \times \widetilde{X} \to G \times X$ is schematically dominant so the map $(\id_G \times \nu)^* : \mathcal{O}(G \times X)^\times \to \mathcal{O}(G \times \widetilde{X})^\times$ is injective. The images of $\alpha^*(\phi)$ and $\pr_2^*(\phi)$ are respectively $\widetilde{\alpha}^*(\widetilde{\phi})$ and $\pr_2^*(\widetilde{\phi})$. Since $\widetilde{\phi}$ is $G$-invariant, we have $\widetilde{\alpha}^*(\widetilde{\phi}) = \pr_2^*(\widetilde{\phi})$, so $\alpha^*(\phi) = \pr_2^*(\phi)$ and thus $\phi \in \mathcal{O}(X)^{\times^G}$.

\medskip
Let $\left((\widetilde{L},\Phi_{\widetilde{L}}),(M,\Phi_M)\right) \in \Pic^G(\widetilde{X}) \times \Pic^G(Z)$. Its image in $\Pic^G(Z)$ is trivial if and only if we have an isomorphism of linearized line bundles $(j^*\widetilde{L}) (\lambda^*M)^{-1} \simeq \widetilde{Z}\times \A^1_k,$ (where $\widetilde{Z}\times \A^1_k$ is endowed with the trivial linearization), that is, if and only if we have an isomorphism of linearized line bundles $\sigma : j^*\widetilde{L} \to \lambda^*M$. Hence, by Proposition \ref{prop_pinching_linearized_line_bundles} again, the sequence is exact at $\Pic^G(\widetilde{X}) \times \Pic^G(Z)$.
\end{proof}

\begin{remark}\label{rem_equiv_Units_Pic}
If $X$ is a proper variety then, since $\mathcal{O}(X) = \mathcal{O}(\widetilde{X})=k$, the equivariant Units-Pic sequence can be simplified as 
$$1 \to \mathcal{O}(\widetilde{Z})^{\times^G} / \mathcal{O}(Z)^{\times^G} \to \Pic^G(X) \to \Pic^G(\widetilde{X}) \times \Pic^G(Z) \to \Pic^G(\widetilde{Z}).$$
\end{remark}

\subsection{Equivariant Picard groups of curves}

We can determine the equivariant Picard group of the almost homogeneous curves classified in Theorem \ref{th_general_classification}. 

\begin{theorem}\label{th_general_PicG}
Let $C$ be a seminormal curve and $G$ a smooth connected algebraic group acting faithfully on $C$.
\begin{enumerate}
\item (homogeneous curves)
\begin{enumerate}
  \item \label{1a}If $C$ is a smooth projective conic and $G \simeq \Autgp{C}$ then $\Pic^G(C) = \Pic C = \Z \cdot \omega_C$ where $\omega_C$ is the canonical sheaf, so $\dfrac{1}{2} \deg : \Pic^G(C) \to \Z$ is an isomorphism.
  \item \label{1b}If $C \simeq \A^1_k$ and $G \simeq \G_{a,k} \rtimes \G_{m,k}$ (acting by affine transformations) then $\Pic^G(C) \simeq \widehat{G}(C) \simeq \Z$.
  \item \label{1c}If $G$ is a form of $\G_{a,k}$ or $\G_{m,k}$ and $C$ is a $G$-torsor then $\Pic^G(C)$ is trivial.
  \item \label{1d}If $C$ is a smooth projective curve of genus $1$ and $G \simeq \Autgp{C}^\circ$ then $\Pic^G(C)$ is trivial.
  \end{enumerate}
\item (regular, non-homogeneous curves)
  \begin{enumerate}
  \item \label{2a}If $C \simeq \PP^1_k$ and $G \simeq \G_{a,k} \rtimes \G_{m,k}$ or $G \simeq \G_{m,k}$ then we have a split exact sequence $$1 \to \widehat{G}(C) \simeq \Z \to \Pic^G(C) \to \Pic C  = \Z \cdot [\infty] \simeq \Z \to 0$$ and the morphism $\Pic^G(C) \to \Z$ is identified with the degree map, so $\Pic^G(C) \simeq \Z^2$.
  \item \label{2b}If $G$ is a form of $\G_{a,k}$, $C$ is the regular completion of a $G$-torsor and $P$ is the point at infinity then we have $\Pic^G(C) = \Z \cdot [P] \simeq \Z$ and this isomorphism is identified with the map $\dfrac{1}{[\kappa(P):k]}\deg$.
  \item \label{2c}If $C \simeq \A^1_k$ and $G \simeq \G_{m,k}$ then $\Pic^G(C) \simeq \widehat{G}(C) \simeq \Z$.
  \item \label{2d}If $C$ is a smooth projective conic and $G$ is the centralizer of a separable point $P$ of degree $2$ then we have a split exact sequence $1 \to \widehat{G}(C) \to \Pic^G(C) \to \Pic C = \Z \cdot [P] \simeq \Z \to 0$ and the morphism $\Pic^G(C) \to \Z$ is identified with $\dfrac{1}{2} \deg$.
  \end{enumerate} 
\item (seminormal, singular, non-homogeneous curves)
  \begin{enumerate}
  \item \label{3a}If $G$ is a non-trivial form of $\G_{a,k}$ and $C$ is obtained by pinching the point at infinity $\widetilde{P}$ of the regular completion $\widetilde{C}$ of a $G$-torsor on a point $P$ whose residue field $\kappa(P)$ is a strict subextension of $\kappa(\widetilde{P})/k$ then we have a split exact sequence $$1 \to \kappa(\widetilde{P})^\times / \kappa(P)^\times \to \Pic^G(C) \to \Z \cdot[\widetilde{P}] \simeq \Z \to 0.$$ and the morphism $\Pic^G(C) \to \Z$ is identified with $\dfrac{1}{[\kappa(\widetilde{P}):k]}\deg$.
  \item \label{3b}If $C$ is obtained by pinching two $k$-rational points of $\PP^1_k$ on a $k$-rational point and $G \simeq \G_{m,k}$ then we have a split exact sequence $$1 \to (k^\times \times k^\times) / k^\times \to \Pic^G(C) \to \widehat{G}(\PP^1_k) \simeq \Z \to 0.$$
  \item \label{3c}If $C$ is obtained by pinching a separable point $\widetilde{P}$ of degree $2$ of a smooth projective conic $\widetilde{C}$ on a $k$-rational point $P$, and $G$ is the centralizer of $\widetilde{P}$ in $\Autgp{\widetilde{C}}$ then we have $\Pic^G(C) \simeq \kappa(\widetilde{P})^\times / \kappa(P)^\times$.
  \end{enumerate}
\end{enumerate} 
\end{theorem}

\begin{proof}
Case $\ref{1a}$ : The group of characters $X(G)$ is trivial so, by Proposition \ref{prop_PicG_exact_seq}, the forgetful morphism $\Pic^G(C) \to \Pic C$ is injective. If $C = \PP^1_k$ and $G = \PGL_{2,k}$ then it is well-known that the sheaf $\mathcal{O}_{\PP^1_k}(1)$ is not linearizable (see \cite[p.33]{Mumford_GIT}), but the canonical sheaf $\omega_{\PP^1_k} = \mathcal{O}_{\PP^1_k}(-2)$ is linearizable. Hence $\Pic^G(C)$ is the subgroup of $\Pic C$ generated by $\omega_{\PP^1_k}$. More generally, let $C$ be a smooth conic (embedded in $\PP^2_k$ by a closed immersion $i$) and $G = \Autgp{C}$. There exists a Galois extension $K/k$ such that $C_K \simeq \PP^1_K$. Then the pullback morphism $\Pic C \to \Pic C_K \simeq \Z$ is injective (its kernel consists of the forms of the trivial line bundle $C_K \times \A^1_K$, so by Galois cohomology and Hilbert's $90$th theorem it is trivial) and is identified with the degree morphism. Thus if $C$ is not isomorphic to $\PP^1_k$ then there is no line bundle of degree one (otherwise $C$ would have a $k$-rational point) so $\Pic C$ is generated by the canonical sheaf $\omega_C$, which is linearizable.

\medskip
Case $\ref{1b}$ : Every line bundle over $C$ is isomorphic to the trivial bundle. By \cite[Prop. 2.10]{BRI_lin}, the kernel of the forgetful morphism $\Pic^G(C) \to \Pic C$ is the group $\widehat{G}(C)$. Moreover, $\widehat{G}$ is the constant sheaf $\underline{\Z}$.

\medskip
Cases $\ref{1c}$ and $\ref{1d}$ : The curve $C$ is a $G$-torsor so, by Lemma \ref{lemma_PicG_torsor}, the group $\Pic^G(C)$ is trivial.

\medskip
Case $\ref{2a}$ : By Proposition \ref{prop_PicG_exact_seq} again, the sequence $1 \to \widehat{G}(C) \to \Pic^G(C) \to \Pic C$ is exact. Let $\alpha : G \times C \to C$ be the standard action and $P = \Spec k$ the reduced closed subscheme of $C$ supported by $\infty$. Denote by $[\infty]$ the corresponding Weil divisor, so that $\Pic C = \Z \cdot [\infty]$. We have the exact sequence of sheaves $$0 \to  \mathcal{O}_C(-P) \to \mathcal{O}_C \to \mathcal{O}_P \to 0.$$ Since $\alpha$ and $\pr_2$ are flat, and $P$ is $G$-stable, we have a commutative diagram 
\begin{center}
\begin{tikzpicture}[baseline=(m.center)]
\matrix(m)[matrix of math nodes,
row sep=2.5em, column sep=2.5em,
text height=1.5ex, text depth=0.25ex,ampersand replacement=\&]
{0 \& \alpha^*\mathcal{O}_C(-P) \& \alpha^*\mathcal{O}_C \& \alpha^*\mathcal{O}_{P} \& 0 \\
0 \& \pr_2^*\mathcal{O}_C(-P) \& \pr_2^*\mathcal{O}_C \& \pr_2^*\mathcal{O}_{P} \& 0 \\};
\path[->] (m-1-1) edge (m-1-2);
\path[->] (m-1-2) edge (m-1-3);
\path[->] (m-1-3) edge (m-1-4);
\path[->] (m-1-4) edge (m-1-5);
\path[->] (m-2-1) edge (m-2-2);
\path[->] (m-2-2) edge (m-2-3);
\path[->] (m-2-3) edge (m-2-4);
\path[->] (m-2-4) edge (m-2-5);
\path[->] (m-1-3) edge (m-2-3);
\path[->] (m-1-4) edge (m-2-4);
\end{tikzpicture}
\end{center}
which is exact in rows and where the vertical arrows are isomorphisms. This yields an isomorphism $\alpha^*\mathcal{O}_C(-P) \simeq \pr_2^*\mathcal{O}_C(-P)$. Since $C$ is reduced, by Remark \ref{rem_linearizable_reduced}, this implies that the line bundle corresponding to $-[P]$ is linearizable. Hence the forgetful morphism 
$\Pic^G(C) \to \Pic C$ is surjective.

\medskip
Case $\ref{2b}$ : As above, the forgetful morphism $\Pic^G(C) \to \Pic C$ is injective. Let $U$ be the open subscheme of $C$ which is a $G$-torsor, so that $P$ is the reduced closed subscheme of $C$ supported by the point $C \setminus U$. If $L$ is a linearizable line bundle over $C$ then its restriction $L_{|U}$ is linearizable. But $U$ is a $G$-torsor so $\Pic^G(U)$ is trivial. So $L$ corresponds to a Weil divisor in $\Z \cdot [P]$. Conversely, by the same argument as above, the line bundle corresponding to $-[P]$ is linearizable. Therefore $\Pic^G(C) = \Z \cdot [P]$.

\medskip
Case $\ref{2c}$ : The argument is the same as for the case $\ref{1b}$.

\medskip
Case $\ref{2d}$ : The canonical divisor on $C$ is $-[P]$ so $\Pic C = \Z \cdot [P]$. Thus the morphism $\Pic^G(C) \to \Pic C$ is surjective.

\medskip
Case $\ref{3a}$ : The curve $C$ is obtained by the pinching diagram
\begin{tikzpicture}[baseline=(m.center)]
\matrix(m)[matrix of math nodes,
row sep=2.5em, column sep=2.5em,
text height=1.5ex, text depth=0.25ex,ampersand replacement=\&]
{\widetilde{Z} \& \widetilde{C} \\
Z \& C\\};
\path[->] (m-1-1) edge (m-1-2);
\path[->] (m-1-1) edge (m-2-1);
\path[->] (m-2-1) edge (m-2-2);
\path[->] (m-1-2) edge (m-2-2);
\end{tikzpicture}
where $Z = \Spec \kappa(P)$ and $\widetilde{Z} = \Spec \kappa(\widetilde{P})$. By Remark \ref{rem_equiv_Units_Pic}, the equivariant Units-Pic sequence can be written as $$1 \to \mathcal{O}(\widetilde{Z})^\times / \mathcal{O}(Z)^\times \to \Pic^G(C) \to \Pic^G(\widetilde{C}) \times \Pic^G(Z) \to \Pic^G(\widetilde{Z}).$$ Since the group of characters $X(G)$ is trivial, it follows from Proposition \ref{prop_PicG_exact_seq} that the groups $\Pic^G(Z)$ and $\Pic^G(\widetilde{Z})$ are trivial. Moreover by the case \ref{2b}, we have $\Pic^G(\widetilde{C}) = \Z \cdot [P]$.

\medskip
Case $\ref{3b}$ : We adapt the argument of \cite[Ex. 2.15]{BRI_lin}. The curve $C$ is obtained by the pinching diagram
\begin{tikzpicture}[baseline=(m.center)]
\matrix(m)[matrix of math nodes,
row sep=2.5em, column sep=2.5em,
text height=1.5ex, text depth=0.25ex,ampersand replacement=\&]
{\widetilde{Z} \& \PP^1_k \\
Z \& C\\};
\path[->] (m-1-1) edge node[above] {$j$}(m-1-2);
\path[->] (m-1-1) edge node[left] {$\lambda$} (m-2-1);
\path[->] (m-2-1) edge node[below] {$i$} (m-2-2);
\path[->] (m-1-2) edge node[right] {$\nu$} (m-2-2);
\end{tikzpicture}
where $\widetilde{Z} = \Spec k \sqcup \Spec k$ is the reduced subscheme of $\PP^1_k$ supported by $\{0,\infty\}$ and $Z = \Spec k$. By Remark \ref{rem_equiv_Units_Pic} again, the equivariant Units-Pic exact sequence can be written as
$$1 \to \mathcal{O}(\widetilde{Z})^\times / \mathcal{O}(Z)^\times = (k^\times \times k^\times) / k^\times \to \Pic^G(C) \to \Pic^G(\PP^1_k) \times \Pic^G(Z) \to \Pic^G(\widetilde{Z}).$$ 
By Proposition \ref{prop_PicG_exact_seq} we have isomorphisms $\Pic^G(\PP^1_k) \simeq \Pic \PP^1_k \times \widehat{G}(\PP^1_k) \simeq \Z \times \Z$, $\Pic^G(Z) \simeq \widehat{G}(Z) \simeq \Z$ and $\Pic^G(\widetilde{Z}) \simeq \widehat{G}(\widetilde{Z}) \simeq \Z \times \Z$. Since $G$ acts on the fibers over $0$ and $\infty$ of a line bundle $\mathcal{O}_{\PP^1_k}(n)$ with respective weights $n$ and $0$, the morphism $\Pic^G(\PP^1_k) \times \Pic^G(Z) \to \Pic^G(\widetilde{Z})$ corresponds to $\fonction{(\Z \times \Z) \times \Z}{\Z \times \Z}{((n,m),\ell)}{(n+m,m)-(\ell,\ell)}$. Its kernel is $\Z \cdot ((0,1),1) \simeq \widehat{G}(\PP^1_k)$. 

\medskip
Case $\ref{3c}$ : Once more, we have the exact sequence 
$$1 \to \kappa(\widetilde{P})^\times / \kappa(P)^\times \to \Pic^G(C) \to \Pic^G(\widetilde{C}) \times \Pic^G(Z) \to \Pic^G(\widetilde{Z})$$
where $Z = \Spec k$ and $\widetilde{Z} = \Spec \kappa(\widetilde{P})$. Let $(L,\Phi_L) \in \Pic^G(C)$ and let us show that its image is trivial. By Proposition \ref{prop_PicG_exact_seq}, we have $\Pic^G(Z) \simeq \widehat{G}(Z) = \Hom_{k-\mathrm{gp}}(G,\G_{m,k}) = \{1\}$. Hence $i^*(L,\Phi_L)$ is trivial, and so $j^*\nu^*(L,\Phi_L)$ is trivial too. The line bundle $L_K$ over $C_K$ is linearizable so, by the case $\ref{3b}$, it has degree $0$. So $L$ and $\nu^*L$ have degree $0$ too. Since $\Pic \widetilde{C} = \Z \cdot \omega_{\widetilde{C}}$, the line bundle $\nu^*L$ is trivial. Its linearization is given by an element $\chi \in \widehat{G}(\widetilde{C})$ (and the linearization of $j^*\nu^*L$ is given by $j^*(\chi)$). The morphism $\widehat{G}(\widetilde{Z}) \to \Pic^G(\widetilde{Z})$ is injective so $j^*(\chi)$ is trivial. After the extension of scalars to $K= \kappa(\widetilde{P})$, we have $\chi_K \in \widehat{G_K}(\widetilde{C}_K)$ whose pullback in $\widehat{G_K}(\widetilde{Z}_K)$ is trivial. But $G_K \simeq \G_{m,K}$ so $\widehat{G_K}$ is the constant sheaf $\underline{\Z}$. Hence $\chi_K$ is trivial, so $\chi$ is trivial too. Therefore the image of $(L,\Phi_L)$ in $\Pic^G(\widetilde{C}) \times \Pic^G(Z)$ is trivial.
\end{proof}

\begin{remark}
Since a line bundle on a projective curve is ample if and only if it has positive degree, we have a full description of ample linearized line bundles over seminormal almost homogeneous curves.
\end{remark}

\begin{corollary}
Let $C$ be a seminormal curve and $G$ a non-trivial smooth connected algebraic group acting faithfully on $C$. Then $C$ can be embedded in the projectivization of a finite-dimensional $G$-module, except in the cases $\ref{1d}$, $\ref{3b}$ and $\ref{3c}$ of Theorem \ref{th_general_PicG}, that is: 
\begin{itemize}
 \item $C$ is a smooth projective curve of genus $1$ and $G \simeq \Autgp{C}^\circ$;
 \item $C$ is obtained by pinching two $k$-rational points of $\PP^1_k$ on a $k$-rational point and $G \simeq \G_{m,k}$;
 \item $C$ is obtained by pinching a separable point $\widetilde{P}$ of degree $2$ of a smooth projective conic $\widetilde{C}$ on a $k$-rational point $P$, and $G$ is the centralizer of $\widetilde{P}$ in $\Autgp{\widetilde{C}}$.
\end{itemize}
\end{corollary}


\bibliographystyle{plain}
\bibliography{AHC}
\addcontentsline{toc}{section}{\protect\numberline{}References}

\end{document}